\documentclass[11pt]{article}%
\usepackage{amsfonts}
\usepackage{amsmath}
\usepackage{amssymb}
\usepackage{graphicx}
\usepackage{geometry}%
\usepackage{color}
\usepackage{epstopdf}
\usepackage[export]{adjustbox}
\usepackage{subcaption}
\usepackage{bm}
\usepackage{mathrsfs}%
\usepackage{bbm}
\usepackage{mathtools}
\usepackage[font={small}]{caption}
\usepackage[titletoc,title]{appendix}
\usepackage{dsfont}
\usepackage{makecell}

\providecommand{\U}[1]{\protect\rule{.1in}{.1in}}
\newtheorem{theorem}{Theorem}

\newtheorem{lemma}{Lemma}

\newtheorem{remark}{Remark}
\newenvironment{proof}[1][Proof]{\noindent\textbf{#1.} }{\ \rule{0.5em}{0.5em}}
\normalsize
\DeclareRobustCommand{\lcroof}[1]{
  \hbox{\vtop{\vbox{%
      \hrule\kern 1pt\hbox{%
        $\scriptstyle #1$%
        \kern 1pt}}\kern1pt}%
    \vrule\kern1pt}}

\begin{document}

\author{Weihong Ni $^1$, Corina Constantinescu $^2$, Alfredo D.~Eg\'{\i}dio dos Reis $^3$ \\
	\& V\'eronique Maume-Deschamps $^{4,}$ \thanks{Authors thank Fundaci\'on MAPFRE, through 2016 Research Grants Ignacio H. de Larramendi, for the financial support given to the project [{\it "TAPAS (Technology Advancement on Pricing Auto inSurance)"}].
}
\thanks{Authors $(^1)$,$(^3)$ gratefully acknowledge support from FCT/MEC (Funda\c c\~ao para a Ci\^encia e a Tecnologia/Portuguese Foundation for Science and Technology) through national funds and when applicable co-financed financed by FEDER, under the Partnership Agreement PT2020}.
\thanks{Authors $(^1)$,$(^4)$ gratefully acknowledge support by the LABEX MILYON (ANR-10-LABX-0070) of Universit\'{e} de Lyon, within the program "Investissements d'Avenir" (ANR-11-IDEX- 0007) operated by the French National Research Agency (ANR)} .
}

\title{Estimation of foreseeable and unforeseeable risks in motor insurance}
\date{}
\maketitle
\begin{center}
	{\it
		$^1$ CEMAPRE, ISEG, Universidade de Lisboa; wni.grace@gmail.com 
	\\
		$^2$  Institute for Financial and Actuarial Mathematics, Department of Mathematical Sciences, University of Liverpool, UK; c.constantinescu@liverpool.ac.uk
		\\
		$^3$ ISEG and CEMAPRE, Universidade de Lisboa, Pt;
		alfredo@iseg.ulisboa.pt
		\\
		$^4$ Institut Camille Jordan,
		Institut de Science Financi\`ere et d’Assurances (ISFA); Universit\'e Claude Bernard Lyon 1, Lyon, Fr, veronique.maume@univ-lyon1.fr
		\\
	}
\end{center}
\vspace{0.5cm}

\begin{abstract}
This project works with the risk model developed by  \cite{li2015risk} and quests modelling, estimating  and pricing insurance for risks brought in by innovative technologies, or other emerging or latent risks. The model considers two different risk streams that arise together, however not clearly separated or observed. Specifically, we consider a risk surplus process where premia are adjusted according to past claim frequencies, like in a Bonus-Malus (BM) system, when we consider a \textit{classical} or \textit{historical risk} stream and an \textit{unforeseeable risk} one. These are unknown risks which can be of high uncertainty that, when pricing insurance (ratemaking and experience rating), suggest a sensitive premium adjustment strategy. It is not clear for the actuary to observe which claim comes from one or the other stream. When modelling such risks it is crucial to estimate the behaviour of such claims, occurrence and their severity. Premium calculation must fairly reflect the nature of these two kinds of risk streams. 

We start proposing a model, separating claim counts and severities, then propose a premium calculation method, and finally a parameter estimation procedure. In the modelling we assume a Bayesian approach as used in credibility theory, a credibility approach for premium calculation and the use of the Expectation-Maximization (EM) algorithm in the estimation procedure. 

\ \\

\end{abstract}

\section{Introduction}
This project aims at developping and also applying statistical tools for parameter estimation into the model introduced by \cite{li2015risk}, where the so-called unforeseeable risks were taken into account. These risks refer to those that are not really predictable due to presence of very little knowledge about them. With the ever smarter way of living, the evaluation of risks could probably be different from classical criteria. For instance, the introduction of autonomous cars would possibly change the magnitude of risks on roads, yet we do not know in which direction precisely, although they were designed to reduce \textit{human error}. To model such unforeseeable feature, a defective random variable was incorporated in the risk model by \cite{li2015risk}. There, unforeseeable risks were reflected in claim frequencies only. Letting premium adjust according to claim numbers imitates the operation of a Bonus-malus system (briefly, BM system or BMS). This project takes the above cited idea, develops their work on the claim frequencies topic, incorporates the two risk stream proposal into the claim severity, add appropriate premium calculation and estimation, and works an estimation procedure, for the distribution parameters of the claim counts and severities, as well as for the premium estimation. For the premium calculation/estimation credibility theory is used, parameter estimation procedure uses the Expectation-Maximization algorithm (briefly EM algorithm). 

EM algortithm is developped in \cite{dempster1977maximum}, a \textit{guided tour} can be found in \cite{couvreur1997algorithm}. We will be using mixture distributions, particular algorithm application for these can be found in 
\cite{redner1984mixture}. For premium calculation we will be using Bayesian credibility methods. Introductory notions on credibility can be found in \cite{klugman2012loss}, Chapters~17-19, and more advanced in \cite{buhlmann2006course}, completely devoted to credibility theory and applications. 

The output of the project should demonstrate a more reasonable pricing model, especially applied for BMS, taking into account those uncertainties (e.g.~emerging risk estimations). By embedding the premium adjustment rules in a classical collective insurance risk model, some new risk measurements will be introduced, so that they can be used in risk management. The work provides a theoretically structured method.

{
The insurance risk process we base our developments is that of defined by Equation~(3) by \cite{li2015risk},  retrieved from \cite{dubey77}, where the Poisson parameter $\lambda$ is a realization of a random variable $ \Lambda $. From there, the premium income along time uses a Bayesian approach and is estimated by the posterior mean $\hat{\lambda}(t)= \mathbb{E}[\Lambda|N(t)] $ (it is one of \cite{dubey77}'s proposed estimates). We mean, the risk process is driven by equation 
\begin{equation}
U(t)=u+c \int_{0}^{t} \hat{\lambda}(s)ds -S(t) \,, \; \; t\geq 0\,, 
\label{eq:model2}
\end{equation}
where, $ S(t)=\sum^{N(t)}_{j=0}Y_j $ is the aggregate claims up to time $t$, $ u=U(0) $ is the initial surplus,  $ \{Y_j\}_{j=1}^{\infty} $ is a sequence of independent and idenciatlly distributed (briefly, i.i.d$ . $) random variables with common distribution $H_Y(.)$, with existing mean $\mu=\mathbb{E}[Y_1]$ and $Y_0 \equiv 0$.  The premium rates are dynamically adjusted, the underlying counting process $ \{N(t), t\geq 0\} $ is a mixed Poisson process with random itensity $ \Lambda$. The other premium comonent, $c$, denotes the time constant part of the premium income, and often defining a loading as a proportion of the pure premium, here $c=(1+\theta)\mathbb{E}[Y_1]$, 
where $\theta$ $(>0)$ is the loading coefficient. 
} 

Another \cite{dubey77}'s neat idea that taken by \cite{li2015risk} was to consider a positive probability  $\mathbb{P}[\Lambda=0]=p\,\, (>0)$, so that the counting process $ \{N(t)\} $ is a mixed Poisson process conditional on $\Lambda>0$. The latter authors generalised to consider that the process $ \{N(t)= N^{(1)}(t)+ N^{(2)}(t)\}$ is a sum of two mixed counting processes, where $ \{N^{(1)}(t)\}$ represents the \textit{historical risk stream} of claim counts and  $ \{N^{(2)}(t)\}$ represents the corresponding \textit{unforeseeable risk stream}. This is the process that we recapture for further developments.  Concerning the first process the randomized intensity parameter has a classical behaviour, it is positive, or  
$\mathbb{P}\{\Lambda^{(1)} = 0\} = 0$, whereas in the second process
and we set
$\mathbb{P}\{\Lambda^{(2)} = 0\} = p>0.$
The first one is assigned to a \textit{historical} risk behaviour, and the second  to an \textit{unforeseeable} risk one. Putting this idea into (\ref{eq:model2}),
\cite{li2015risk} set $\hat{\lambda}(t)= \mathbb{E}[\Lambda^{(1)}+\Lambda^{(2)}|N(t)] $.

Modelling in this way is assuming that the second stream may be either not as risky as the classical one, or much riskier. If we think of the autonomous cars, experiments to date have suggested that they are potentially safer than human driving (Waymo Team, 2016\nocite{waimo216}). On the other hand, if we think of the asbestos problem it may be the other way around. Not much interpretation is given to the situation where $p>0$. About the possible outbreak of unforeseeable clims, we quote \cite{li2015risk} at the start of their Section~2 where they say that
\begin{quotation}
\textit{...once they broke out in a negative way, it would possibly be too late for an insurance company to control the losses... the expectation of $\Lambda$ conditioning on $\{\Lambda >0\}$ is very large and practically it could be assumed that it could be much more than the average number of claims in the ``historical'' stream.}
\end{quotation} 

Modelling two different streams of risks for the same portfolio can be done exclusively on the claim count process like \cite{dubey77} and \cite{li2015risk} did, or also in the claim severity if we consider dependence between $ \{ N(t), t\geq 0 \} $
and sequence $ \{ Y_j \}^{\infty}_{j=1} $. This latter means that the different streams may (also) \textit{bring} different severity behaviour. This is not dealt by the previous authors. We start dealing with the model as the above authors did, we mean work first the claim frequency component separately. \cite{li2015risk} were more concerned in ruin probability calculation, embodying the effect of the unforesseable stream in the claim counts only.

The manuscript is organised as follows. In the next section we work the claim frequency component of the model, in Section~\ref{s:sev_comp} we complete the starting model with the introduction of claim severity, since we admit that the severity may bring some information on the stream origin. In Section~\ref{sec:Bayesian} we deal with the premium calculation/estimation, where a Bayesian credibility approch is used. Section~\ref{s:estima} is devoted to parameter estimation where the EM algortithm is used, dealing separately the claim count and severity estimation. We finish  the section by adding some discussion on the estimation procedure and results. Our manuscript is finished inserting some notes on how to do future estimation by setting a \textit{global} likelihood in order to estimate all parameters together, from both claim frequency and severity distributions, regarding the application of estimation algorithms like the EM algorithm.
\section{Claim frequency component} \label{s:claim_comp}
The total claim numbers over a time period $(0,t]$ within the portfolio is given by 
\begin{equation}
\label{eq:jointprocess}
N(t) = N^{(1)}(t)+N^{(2)}(t)\,.
\end{equation}
Due to the randomness in their intensity parameters, we know that $\{N^{(1)}(t)\}$ and $\{N^{(2)}(t)\}$ are both mixed counting processes. More specifically, we consider now that $\{N^{(1)}(t)\}$ is a mixed Poisson process, whereas $\{N^{(2)}(t)\}$ is mixed Poisson conditional on $\{\Lambda^{(2)} > 0\}$, since the intensity must be positive. We consider that the event $\{\Lambda^{(2)} = 0\}$ implies that $\{N^{(2)}(t) = 0\}$, with probability one, like behaving as a limiting situation of a Poisson intensity where the probability
  $$\lim_{\lambda\downarrow 0}\mathbb{P}\{N^{(2)}(t)=0|\Lambda=\lambda\}=\lim_{\lambda\downarrow 0}e^{-\lambda\,t}=1\,.$$ 
It was proved by \cite{li2015risk} [see their Lemma~1] that $\{N(t)\}$ is also a mixed Poisson process with intensity $\Lambda^{(1)}+\Lambda^{(2)}$ given that $N^{(1)}(t)$ and $N^{(2)}(t)$ are independent. Under this global process, $\{N(t)=N^{(1)}(t)+N^{(2)}\}$, we have that
when a claim arrives it comes either from process $ \{N(t)^{(1)}\}$,   
or 
from process 
$\{N(t)^{(2)}\}$. We remark that in practice actuaries will mostly observe the realization of the global claim process, which means they'll have to \textit{guess} which process the arriving claim belongs to. 

Let's define and denote $\Xi$ as the \textit{split rate}, in favour of process $\{N(t)^{(1)}\}$ in $\{N(t)\}$, between the two previous processes. It is well known that for independent Poisson processes, in our case for given and positive rates $\Lambda^{(1)}=\lambda^{(1)}$ and $\Lambda^{(2)}=\lambda^{(2)}>0$, we have that the \textit{split} between processes ``(1)'' and ``(2)'' in the combined process is going to be given by the probabilities, respectively as 
\begin{eqnarray}
\label{eq:epsilon}	
\xi&=&\dfrac{\lambda^{(1)}}{\lambda^{(1)}+\lambda^{(2)}} \;\;\;\, 
 \text{and} \;\;\;\,	(1- \xi) 
 = 
 \dfrac{\lambda^{(2)}}{\lambda^{(1)}+\lambda^{(2)}}\,. 
\end{eqnarray}
 Unconditionally, considering that $\xi$ is particular outcome of the random variable $\Xi$, then we can write, \textit{extending} to the unconditional mixed process as defined in \eqref{eq:jointprocess}: 
 \begin{eqnarray}	
 \label{eq:Xi}
 \Xi&=&\dfrac{\Lambda^{(1)}}{\Lambda^{(1)}+\Lambda^{(2)}} \,.  
  \end{eqnarray}
 Note that
  \begin{equation}
  \label{eq:probxi1}
 \{\Lambda^{(2)}= 0\} \iff \{\Xi=1\}  \Rightarrow \mathbb{P}[\Lambda=0]=\mathbb{P}[\Xi=1]=p\, .
 \end{equation}
 We will return to this issue later in the text. We denote the distribution function of $\Xi$ as $A_\Xi(.)$. 
 
 We will now set further assumptions. From now onwards, we assume assume that the Poisson intensities follow gamma distributions: $\Lambda^{(1)}\sim Gamma(\alpha_1, \beta_1)$ and, conditional on $ \Lambda^{(2)}>0 $, $\Lambda^{(2)}|\Lambda^{(2)}>0\sim Gamma(\alpha_2, \beta_2)$. For simplicity, in the sequel we denote $\Lambda^{(2)}_+=\Lambda^{(2)}|\Lambda^{(2)}>0$. Also, without loss of generality, we will be working with the random variable $N(1)=N$, denoting $N$ for simplicity. Likewise with related quantities, such as $N^{(1)}(1)=N^{(1)} $ and $ N^{(2)}(1)=N^{(2)} $.

Under a more restrict assumption for the gamma distributions above, we can arrive at the following result:
\begin{lemma} \label{l:mix_nb}
If $\beta_1 = \beta_2 = \beta$, then $N(1)=N$ is a mixture of two Negative Binomial random variables, such that 
$$\mathbb{P}(N = n) = p\binom{n+\alpha_1-1}{n}\left(\frac{\beta}{\beta + 1}\right)^{\alpha_1}\left(\frac{1}{\beta + 1}\right)^n + (1-p)\binom{n+\alpha_1+\alpha_2-1}{n}\left(\frac{\beta}{\beta + 1}\right)^{\alpha_1+\alpha_2}\left(\frac{1}{\beta + 1}\right)^n.$$
\end{lemma}
\begin{proof}
Let $ N^{(1)}(1)=N^{(1)} $ and $ N^{(2)}(1)=N^{(2)} $, for simplicity. Taking the moment generating function (briefly, mgf) of $N$, as function of $\rho$,  yields
\begin{eqnarray*}
M_N(\rho) &=& \mathbb{E}\left[e^{\rho\,(N^{(1)}+N^{(2)})}\right] = M_{N^{(1)}}(\rho)M_{N^{(2)}}(\rho)\\
&=&\left(\frac{1-\frac{1}{1+\beta}}{1-e^\rho \frac{1}{1+\beta}}\right)^{\alpha_1}\left[p+(1-p)\left(\frac{1-\frac{1}{1+\beta}}{1-e^\rho \frac{1}{1+\beta}}\right)^{\alpha_2}\right]\\
&=&p\left(\frac{1-\frac{1}{1+\beta}}{1-e^\rho \frac{1}{1+\beta}}\right)^{\alpha_1}+(1-p)\left(\frac{1-\frac{1}{1+\beta}}{1-e^\rho \frac{1}{1+\beta}}\right)^{\alpha_1+\alpha_2}.
\end{eqnarray*}
We know that this corresponds to a weighted average of two Negative Binomial distributions, with weights $p$ and $1-p$ respectively. \hfill{}
\end{proof}
\begin{remark}
Note that $N^{(2)}$ follows a `Zero Modified' Negative Binomial distribution, briefly, $N^{(2)}\frown ZM\,\, Negative\text{ } Binomial (\alpha_2,\beta)$. It belongs to the  $(a,b,1)$ recursion class of distributions, see Section~6.6 of \cite{klugman2012loss}. Since
\begin{eqnarray*}
M_{N^{(2)}}(\rho) &=
& \mathbb{E}\left[e^{N^{(2)}}|\Lambda^{(2)}=0\right]\times \mathbb{P}\left[ \Lambda^{(2)}=0
 \right]	
+ \mathbb{E}\left[e^{N^{(2)}}|\Lambda^{(2)}>0\right]\times \mathbb{P}\left[ \Lambda^{(2)}>0\right]\\
 &=
& p+ (1-p) \left(\frac{1-\frac{1}{1+\beta}}{1-e^\rho \frac{1}{1+\beta}}\right)^{\alpha_2}\,,
\end{eqnarray*}
which is a weighted average of the mgf's of  a degenerate distribution at $\{0\}$ and a Negative Binomial, member of the $(a,b,0)$ class. \hfill{$\Box$}
\end{remark}

Alternatively, instead of focusing on the claim arrival process $\{N(t)\}$ which consists of combining two counting processes, $\{N^{(1)}(t)\}$ and $\{N^{(2)}(t)\}$, as explained above, we could model the underlying claim counts via mixing random variables, mixing two Gamma random variables with an independent Bernoulli random variable.
\begin{lemma} 
Assume that the randomized Poisson parameter $\Lambda$ follows a prior distribution as a mixture of two Gamma random variables, such that
\begin{equation}
\Lambda = I\, Z_1 + (1-I)\, Z_2,
\end{equation}
where $Z_1\sim Gamma(\alpha_1, \beta)$ and $Z_2\sim Gamma(\alpha_1+\alpha_2, \beta)$ and $I$ is a Bernoulli$(p)$ random variable independent of $Z_1$ and $Z_2$.
Then the marginal distribution of $N(t) = N$ for a fixed $t$ is a mixture of two Negative Binomial random variables as shown in Lemma~\ref{l:mix_nb} above. \hfill{}
\label{l:mix_nb2}
\end{lemma}
\begin{proof}
Recall that given $\Lambda = \lambda$,
$
\mathbb{P}\{N=n |\Lambda =\lambda\} = {e^{-\lambda}\lambda^n}/{n!}
$.
Now we integrate over $\Lambda$ whose density, denoted as $\pi(.)$, can be written as
\begin{equation}
\label{eqn:gammamix}
\pi(\lambda) = p\cdot\frac{\lambda^{\alpha_1-1}\beta^{\alpha_1} e^{-\beta \lambda}}{\Gamma(\alpha_1)} + (1-p)\cdot\frac{\lambda^{(\alpha_1+\alpha_2)-1}\beta^{\alpha_1+\alpha_2} e^{-\beta \lambda}}{\Gamma(\alpha_1+\alpha_2)}.
\end{equation}
Hence,
\begin{eqnarray*}
\mathbb{P}\{N=n\} &=& \int_\Lambda \frac{e^{-\lambda}\lambda^n}{n!} \left(p\cdot\frac{\lambda^{\alpha_1-1}\beta^{\alpha_1} e^{-\beta \lambda}}{\Gamma(\alpha_1)} + (1-p)\cdot\frac{\lambda^{(\alpha_1+\alpha_2)-1}\beta^{\alpha_1+\alpha_2} e^{-\beta \lambda}}{\Gamma(\alpha_1+\alpha_2)}\right)\,\rm d\lambda\\
&=& p \cdot\int_\Lambda \frac{e^{-\lambda}\lambda^n}{n!}\frac{\lambda^{\alpha_1-1}\beta^{\alpha_1} e^{-\beta \lambda}}{\Gamma(\alpha_1)}\,\rm d\lambda + (1-p)\cdot\int_\Lambda \frac{e^{-\lambda}\lambda^n}{n!}\frac{\lambda^{(\alpha_1+\alpha_2)-1}\beta^{\alpha_1+\alpha_2} e^{-\beta \lambda}}{\Gamma(\alpha_1+\alpha_2)}\,\rm d\lambda\\
&=&p\binom{n+\alpha_1-1}{n}\left(\frac{\beta}{\beta + 1}\right)^{\alpha_1}\left(\frac{1}{\beta + 1}\right)^n\\
&& + (1-p)\binom{n+\alpha_1+\alpha_2-1}{n}\left(\frac{\beta}{\beta + 1}\right)^{\alpha_1+\alpha_2}\left(\frac{1}{\beta + 1}\right)^n \,.
\end{eqnarray*}
This distribution coincides with the one shown in Lemma~\ref{l:mix_nb}. \hfill{}
\end{proof}

The above enhances the fact that the distribution for claim counts $N(t)$ over a fixed period $t$ is equivalent to a mixing over  a Poisson with Gamma mixtures. In this way, under a Bayesian set-up, considering a prior distribution of Gamma mixtures, we will be able to compute the posterior distribution. This is important for computing estimates aiming experience rating. We show this in Section~\ref{sec:Bayesian}.
%
In fact, the equivalence of the previous two constructions can be explained further. Consider the following remark.
\begin{remark}
	Let $ Z_1=\Lambda^{(1)} $, $ Z_2 = \Lambda^{(1)}+\Lambda^{(2)}_+ $, where $\Lambda^{(2)}_+=\Lambda^{(2)}|\Lambda^{(2)}>0$. Also,  define $\Lambda^{(2)} = (1-I)\,\Lambda^{(2)}_+ + I \Lambda^{(2)}_0= (1-I)\,\Lambda^{(2)}_+$, with $\Lambda_0^{(2)} = 0$ and $I\sim \text{Bernoulli}(p)$. 
	We have,
\begin{eqnarray*}
\Lambda &=& I\, Z_1 + (1-I)\, Z_2 = I\,\Lambda^{(1)}+(1-I)\,(\Lambda^{(1)}+\Lambda^{(2)}_+) \\
&=& \Lambda^{(1)} + (1-I) \,\Lambda^{(2)}_+ = \Lambda^{(1)}+\Lambda^{(2)}\,.
\end{eqnarray*}
Since $N|\Lambda \sim$ \text{Poisson} $(\Lambda^{(1)}+\Lambda^{(2)})$, it is true that $N = N^{(1)}+N^{(2)}$ where $N^{(1)}|\Lambda^{(1)}\sim  \text{Poisson}(\Lambda^{(1)})$ and $N^{(2)}|\Lambda^{(2)}_+\sim \text{Poisson}(\Lambda^{(2)}_+)$. \hfill{$\Box$}
\end{remark}
%
%
%
In the model presented so far, we distinguish the \textit{unforeseeable} stream of risks from the \textit{historical} stream by considering two different, and independent, claim counting processes where the randomized intensities are of different nature. 
%
For now, these parameters only cause influence in the claim number process, not the claim size. 
However, we should also consider the possibility that the amounts are affected by the risk type. The introduction of a randomness itself in the intensities does not necessarily bring a dependence between the number of claims and sizes. However, as we put together the two risk streams in a global process, this is the one observed, and we consider that the claim size behaviour may depend on the risk stream, we need to admit dependence. This sets us away from the classical risk consideration that claim arrival process is independent of the claim amount one. We may assume the twofold stage: 
\begin{enumerate}
\item On a starting stage, we can consider that for a given $\Lambda^{(i)}$, $i=1,2$, the independence is only (conditional) inside each ($i$). Or,%
\item On a further stage, a model with \textit{extended} dependence.
\end{enumerate} 
The first situation seems quite feasible for our model, if we consider no independence. Our next developments will clear the situation.
We remark that parameters cannot be observed, only claims can. Further, we, as actuaries, only observe the realization of the total claim process $\{N(t)=N^{(1)}(t)+N^{(2)}(t)\}$, from there we'll have to \textit{guess}/estimate which process each claim arrival belongs to. Then, using data and the model as behaving like is defined in Lemma~\ref{l:mix_nb} we have four parameters to estimate, $p$, $\alpha_1$, $\alpha_2$, and $\beta$. 

As defined in \eqref{eq:Xi}, $\Xi$ is viewed as the (random) split rate for process between processes $\{N^{(1)}(t)\}$ and $\{N^{(2)}(t)\}$ in $\{N(t)\}$. Let's consider the event $\{\Lambda^{(1)} + \Lambda^{(2)} = \lambda,\,  \Xi = \xi ,\, 0<\xi<1 \}$ as given, so that $\{N(t)\}$ conditionally is a Poisson process. Consider the following lemma:
\begin{lemma}
	Under $\{\Lambda^{(1)} + \Lambda^{(2)} = \lambda, \Xi = \xi, 0<\xi<1 \}$, the split rate, of process $\{N(t)=N^{(1)}(t)+N^{(2)}(t)\}$ in favour of $\{ N^{(1)}(t) \}$,  is given by $\xi = \frac{\Lambda^{(1)}}{\Lambda^{(1)}+\Lambda^{(2)}}$. 
\end{lemma}
\begin{proof}
	Let us assume there are $n=n_1+n_1$ number of arrivals within a fixed time $(0,t]$ where $n_1$ came from the historical stream and $n_2$ arrived from the unforeseeable stream. We could thus write, conditioning on $\{\Lambda^{(1)} + \Lambda^{(2)} = \lambda, \Xi = \xi\}$
	\begin{eqnarray*}
		&& \!\!\!\!\!\! \!\!\!\!\!\! \!\!\!\!\!\! \!\!\!\!\!\!
		\mathbb{P}\left\{N^{(1)}(t) = n_1, N^{(2)}(t) = n_2 \middle| \Lambda^{(1)} + \Lambda^{(2)} = \lambda, \Xi = \epsilon\right\} \\
		&=& \sum_{n = 0}^{\infty}\mathbb{P}\left\{ N^{(1)}(t) = n_1, N^{(2)}(t) = n_2 \middle | N^{(1)}(t) + N^{(2)}(t) = n, \Lambda^{(1)} + \Lambda^{(2)} = \lambda, \Xi = \xi\right\}\\
		&& \;\;\; \times \, \mathbb{P}\left\{N^{(1)}(t) + N^{(2)}(t) = n \middle | \Lambda^{(1)} + \Lambda^{(2)} = \lambda, \Xi = \xi \right\}.
	\end{eqnarray*}
	But $\mathbb{P}\left\{ N^{(1)}(t) = n_1, N^{(2)}(t) = n_2 \middle | N^{(1)}(t) + N^{(2)}(t) = n, \Lambda^{(1)} + \Lambda^{(2)} = \lambda, \Xi = \xi\right\} = 0$ when $n\neq n_1 + n_2$, then equation above continues as
	\begin{eqnarray*}
		&&   \!\!\!\!\!\! \!\!\!\!\!\! \!\!\!\!\!\! \!\!\!\!\!\!
		\mathbb{P}\left\{N^{(1)}(t) = n_1, N^{(2)}(t) = n_2 \middle| \Lambda^{(1)} + \Lambda^{(2)} = \lambda, \Xi = \epsilon\right\} \\
		&=&\mathbb{P}\left\{ N^{(1)}(t) = n_1, N^{(2)}(t) = n_2 \middle | N^{(1)}(t) + N^{(2)}(t) = n_1+n_2, \Lambda^{(1)} + \Lambda^{(2)} = \lambda, \Xi = \xi\right\}\\
		&&\times\, \mathbb{P}\left\{N^{(1)}(t) + N^{(2)}(t) = n_1+n_2 \middle | \Lambda^{(1)} + \Lambda^{(2)} = \lambda, \Xi = \xi \right\} \, .
	\end{eqnarray*}
The second probability on the right-hand side of the expression above does not depend on $\xi$, thus
\begin{eqnarray*}
	\mathbb{P}\left\{N^{(1)}(t) + N^{(2)}(t) = n_1+n_2 \middle | \Lambda^{(1)} + \Lambda^{(2)} = \lambda, \Xi = \xi \right\}	&=& \frac{(\lambda t )^{n_1+n_2}e^{-\lambda t}}{(n_1+n_2)!}.
\end{eqnarray*}
	Recall that under $\{\Lambda^{(1)} + \Lambda^{(2)} = \lambda, \Xi = \xi\}$, we have $\xi$ denoting the \textit{split rate} for the underlying Poisson process with parameter $\lambda$. Precisely, $\xi$ means the probability of drawing an event from the first stream. Hence, given that $n_1+n_2$ events occurred, each event has probability $\xi$ of being a Stream 1 event and probability $1-\xi$ of being a Stream 2 event, then we have a binomial distributed event coming into play. 
	Hence, the first probability in the same (above) expression is given by 
	\begin{eqnarray*}
	& &  \!\!\!\!\!\! \!\!\!\!\!\! \!\!\!\!\!\! \!\!\!\!\!\!
	\mathbb{P}\left\{ N^{(1)}(t) = n_1, N^{(2)}(t) = n_2 \middle | N^{(1)}(t) + N^{(2)}(t) = n_1+n_2, \Lambda^{(1)} + \Lambda^{(2)} = \lambda, \Xi = \xi\right\}\\
	&=& \binom{n_1+n_2}{n_1}\xi^{n_1}(1-\xi)^{n_2}\,.
	\end{eqnarray*}
	Hence, 
	\begin{eqnarray*}
		\mathbb{P}\left\{N^{(1)}(t) = n_1, N^{(2)}(t) = n_2 \middle| \Lambda^{(1)} + \Lambda^{(2)} = \lambda, \Xi = \xi\right\}
		&=& \binom{n_1+n_2}{n_1}\xi^{n_1}(1-\xi)^{n_2}\,
		\frac{(\lambda t )^{n_1+n_2}e^{-\lambda t}}{(n_1+n_2)!}
	\end{eqnarray*}
	
	With simple calculation, noting that we can write $\lambda=\lambda\, \xi+ \lambda\, (1-\xi)$, we have
		\begin{eqnarray*}
		\mathbb{P}\left\{N^{(1)}(t) = n_1, N^{(2)}(t) = n_2 \middle| \Lambda^{(1)} + \Lambda^{(2)} = \lambda, \Xi = \xi\right\}
		&=& \left( \frac{e^{-\lambda \xi t} (\lambda \xi t)^{n_1}}{n_1 !}\right)
		\left( \frac{e^{-\lambda (1-\xi) t} (\lambda (1-\xi) t)^{n_2}}{n_2 !}\right).
	\end{eqnarray*}
We have independence between two Poisson random variables with means $\lambda^{(1)}t=\lambda  \xi t$ and $\lambda^{(2)}t=\lambda \, (1-\xi) \, t$, and $\xi$ is such that  
\begin{equation*}
 \xi =\dfrac{\lambda^{(1)}}{\lambda}	= \frac{\lambda^{(1)}}{\lambda^{(1)}+\lambda^{(2)}} \,.
\end{equation*}
 \hfill
\end{proof}
\begin{remark}
Remark that $\Xi=\xi$ is the probability that an event arriving from process
	$\{N(t)=N^{(1)}+N^{(2)}\}$ is generated by 	$\{N^{(1)}\}$. If $\{\Lambda^{(1)} + \Lambda^{(2)} = \lambda, \Xi = \xi\}$ is given then  $\Lambda^{(1)}$ is also given since $ \Lambda^{(1)}=\lambda \xi $. Then  $ \Lambda^{(2)}=\lambda - \Lambda^{(1)} = \lambda(1-\xi)\,$.
If $\xi=1 \Rightarrow \Lambda^{(2)}=0 \Leftrightarrow \Lambda= \Lambda^{(1)}=\lambda$ and $N^{(2)}(t)=0 \; \forall\, t\geq 0$ almost surely.
In addition, since the values of $\lambda, \xi, \lambda^{(1)}$ are arbitrary, It should be also true that
$$(\Lambda^{(1)} + \Lambda^{(2)})\Xi = \Lambda^{(1)} \Leftrightarrow \Xi = \frac{\Lambda^{(1)}}{\Lambda^{(1)}+\Lambda^{(2)}}\, .$$
 The variable $\Xi$ represents the random split rate of process $\{N(t)\}$ in favour of process $\{N^{(1)}(t)\}$ previously defined in the beginning of this section.   \hfill{$\Box$}
\end{remark}

In the next section we are going to model the influence of each risk stream in each claim severity.
\section{Claim severity component \label{s:sev_comp}}
For now, we are going to admit that the distribution of the individual claim size depend on the stream type,  either \textit{historical} or \textit{unforeseeable}. We keep presuming that the actuary may not be able to disclose which stream the claim comes from. At least, he cannot be certain.

Let $F$ and $G$ denote the distributions for claim severities in the \textit{historical} and \textit{unforeseeable} streams, respectively. Consider that the individual claim size, taken at random,  say $Y$, follows a distribution function denoted as $H(y)$.
\begin{lemma}
	For a given claim $Y$, its distribution function, conditional on $\Xi=\xi$, can be represented by
	\begin{equation}
	\label{eq:dfYxi}
	\mathbb{P}\{Y\leq y|\Xi=
\xi \}:= H_{\xi}(y) = \xi F(y) + (1-\xi)G(y),
	\end{equation}
	where $\Xi = \frac{\Lambda^{(1)}}{\Lambda^{(1)}+\Lambda^{(2)}}\in(0,1]$, and $F, G$ correspond to the distributions for claim severities in the historical and unforeseeable streams, respectively. The set $\{\Xi = 1\} = \{\Lambda^{(2)} = 0\}$ has a probability measure $p$.
\end{lemma}
\begin{proof}
	The proof is straightforward using the law of total probability.
	\begin{eqnarray*}
		\mathbb{P}\{Y\leq y|\xi\} &=& \mathbb{P}\{Y\leq y | Y = Y^{(1)}, \xi\}\mathbb{P}\{Y = Y^{(1)}| \xi\} + \mathbb{P}\{Y\leq y | Y = Y^{(2)}, \xi\}\mathbb{P}\{Y = Y^{(2)}| \xi\}\\
		&=& F(y)\mathbb{P}\{Y = Y^{(1)}|\xi \} + G(y)\mathbb{P}\{Y = Y^{(2)}| \xi\}\\
		&=&F(y)\cdot\xi + G(y)\cdot(1-\xi).
	\end{eqnarray*}
	where $Y^{(1)}$ and  $Y^{(2)}$ denote random variables of the size of claims stemming from Stream~1 (the \textit{historical stream}) and Stream~2 (the \textit{unforeseeable stream}) respectively.  Note that when $\Xi = 1$, i.e., $\Lambda^{(2)}=0$, $N(t) = N^{(1)}(t)$ and the probability of having a claim from Stream 1, i.e., $\mathbb{P}\{Y = Y^{(1)}|\xi =1 \} = 1$, where as $\mathbb{P}\{Y = Y^{(2)}| \Xi=1\}=0$. This special situation does not affect Equality~\eqref{eq:dfYxi}. \hfill
\end{proof}

\

It has been illustrated in \cite{li2015risk}, see their Lemmas~2 and~3,  that the independence between $\frac{\Lambda^{(1)}}{\Lambda^{(1)}+\Lambda^{(2)}_+}$ 
and $\Lambda^{(1)}+\Lambda^{(2)}_+$, 
conditional on $\Xi \neq 1$, 
result in $\Xi | {\Xi \neq 1}\,$ being distributed as  a $Beta (\alpha_1, \alpha_2)$ law. Then, it is easy to derive the unconditional distribution of $Y$ under these assumptions.
\begin{lemma}
\label{lem:claimsize}
Assume that $\frac{\Lambda^{(1)}}{\Lambda^{(1)}+\Lambda^{(2)}_+}$ and $\Lambda^{(1)}+\Lambda_+^{(2)}$ are independent, and that $\Lambda^{(1)} \sim Gamma (\alpha_1, \beta)$ and $\Lambda_+^{(2)} \sim Gamma(\alpha_2, \beta)$. A given claim size is distributed according to a mixture law of $F(\cdot)$ and $G(\cdot)$, i.e., with density,
\begin{equation}
h_Y(y) = \nu\cdot f(y) + (1-\nu)\cdot g(y),
\label{eqn:clmdist}
\end{equation}
where $\nu = p+(1-p)\frac{B(\alpha_1+1, \alpha_2)}{B(\alpha_1, \alpha_2)}$ and $B(\cdot, \cdot)$ being the Beta function.
\end{lemma}
\begin{proof}
Under the assumptions, we know that $\Xi| {\Xi\neq 1} \frown Beta (\alpha_1, \alpha_2)$ according to \cite{lukacs1955characterization}'s proportion-sum independence theorem. The subsequent proof is then straightforward.  $A_\Xi (\xi)$ denotes the distribution function of $\Xi$  (a mixture), we have
\begin{eqnarray*}
	\mathbb{P}\{Y\leq y\} &=& \int_{(0,1]} H_{\xi}(y) d A(\xi) = pF(y) + (1-p)\int_{(0,1)}H_{\xi}(y)\cdot \frac{\xi^{\alpha_1-1}(1-\xi)^{\alpha_2-1}}{B(\alpha_1, \alpha_2)}d\xi\\
	&=&pF(y) + (1-p)\left[\frac{B(\alpha_1+1, \alpha_2)}{B(\alpha_1, \alpha_2)}\cdot F(y) + \frac{B(\alpha_1, \alpha_2+1)}{B(\alpha_1, \alpha_2)}\cdot G(y)\right]\\
	&=&\left[p+(1-p)\frac{B(\alpha_1+1, \alpha_2)}{B(\alpha_1, \alpha_2)}\right]\cdot F(y) + (1-p) \frac{B(\alpha_1, \alpha_2+1)}{B(\alpha_1, \alpha_2)}\cdot G(y)\\
	&=& \nu\cdot F(y) + (1-\nu)\cdot G(y)\,.
\end{eqnarray*}
 \hfill
\end{proof}

\begin{remark}
	\label{rem:indep}
	If we look carefully at the end of the proof of \cite{li2015risk}' Lemma~2 (see the Appendix, it is not clear from the Lemma) we can conclude that, under the conditions of Lemma~\ref{lem:claimsize}, $Y$ is independent from $N$. This will allow us to consider the parameter estimation using separately the claim frequency and the severity component.
\end{remark} 

\medskip

In the previous section we specified a distribution for the random variable $N(t)$, as well as its parametrization. The unconditional distribution for $N(t)$ was found starting from the Poisson distribution, this is commonly accepted as a starting assumption for claim count data in actuarial science, particularly in motor insurance. We ended up developing a particular mixture of two Negative Binomial distributions for $N(t)$, having started from a different Poisson distribution for each stream of risks.  We'll see in Section~\ref{s:estima} that the assumption is reasonable and may fit actual data.
For the claim severity we will consider a similar reasoning, basing in a particular case from the Gamma distribution family. The Gamma distribution family is also a common fit for insurance severity data. First, we assume a particular form for the distribution of the individual claim severity as a mixture of two distributions, denoted as $F(\cdot)$ and $G(\cdot)$. They represent the behaviours of the claim severities from the two streams, separately the historical and unforeseeable streams, respectively $F(\cdot)$ and $G(\cdot)$. Then, if we specify $F(\cdot)$ and $G(\cdot)$, we can move for the parameter estimation.  

Under many different possibilities we took a simple case but an illustrative one. Consider that both the densities of the two streams come from a common ground, and that in general the claim sizes are distributed as the exponential distribution.
 $Y\sim Exp(\Theta)$, where $\Theta^{-1}$ is the mean, and there is a dependence structure embedded in a random parameter  $\Theta$, such that: 
\begin{enumerate}
	\item 	$\Theta = \mu$ with probability~one if we consider the historical stream; and, 
	\item 
	$\Theta \sim Gamma(\delta, \sigma)$ if we consider the unforeseeable one.
\end{enumerate}
This implies that claims  in the historical stream conform to the Exponential distribution with mean $\mu^{-1}$, whereas those of the unforeseeable stream conform to a Pareto distribution, $Pareto(\delta, \sigma)$. Their respective unconditional densities are
\begin{equation}
\label{eq:sevdist}
f(y) =  \mu e^{-\mu y} \quad \mbox{and}\quad g(y) = \frac{\delta \sigma^\delta}{(\sigma+y)^{\delta+1}}, \quad \mu, \delta, \sigma > 0\, .
\end{equation}

Our model will be complete, and ready for parameter estimation/distribution fit, once the premium calculation is defined. It is of course essential for the risk modelling as well as for risk pricing. So, in the next section we're going to address the premium calculation. Previous developments lead us naturally to deal premium calculation under a Bayesian framework. As explained in the above we will work separately the claim counts and severity premium calculation.
\section{Bayesian premium} \label{sec:Bayesian}
A Bayesian approach and credibility in premium calculation is naturally opened. We'll work with risk random variables that are observable and these variables  follow  distributions that are mixtures of other distributions.  In turn, these latter distributions are assumed to be members of the linear exponential family. Associated to these we consider conjugate priors as distributions for the parameters in a Bayesian setting. Under this assumption credibility is shown to be exact, i.e., the credibility premium coincides with the Bayesian premium. In another words, the Bayesian premium is a linear function with the respect to the observable risk data, making premium calculation straightforward.

Retrieving some essentials in the theory of credibility premium, particularly \textit{credibility exact}, consider a generic a random variable, say $X$, following a distribution belonging to the linear exponential family, depending of some parameter, say $\theta$. Its probability or density function is of the form [see \cite{klugman2012loss}, Chapters~5 and 15]:
\begin{equation}
\label{eq:lexpof}
f_{X|\Theta}(x|\theta) = \frac{p(x) e^{r(\theta) x}}{q(\theta)}\, .
\end{equation}
If parameter $\theta$ is considered to be an outcome of a random variable $\Theta$, whose density function is of the following form, denoted as $\pi(.)$:
\begin{equation}
\label{eq:lexpoprior}
\pi(\theta) = \frac{[q(\theta)]^{-k} e^{\mu k r(\theta) } r'(\theta)}{c(\mu, k)}\,,
\end{equation} 
then we are in the presence of a conjugate prior. In the above formulae, \eqref{eq:lexpof} and \eqref{eq:lexpoprior}, $p(.)$ is some function not depending of parameter $\theta$, $r(.)$, $q(.)$ are functions of the parameter, and $c(.)$ is a normalizing function of given parameters.

In this manuscript we are working with finite mixtures of prior distributions of form \eqref{eq:lexpoprior} and distributions of the linear exponential family \eqref{eq:lexpof}. Recall that in the case of the claim counts we considered a Poisson random variable mixed with the mixture given by \eqref{eqn:gammamix}, and that in the case of the claim severity we considered an exponential mixed with a prior that can be given as a mixture of a degenerate distribution and a Gamma (see end of Section~\ref{s:sev_comp}).

Consider from now on that the prior distribution is of the following form: 
\begin{equation}
\pi(\theta) = \sum_{i = 1}^\eta \omega_i \frac{[q(\theta)]^{-k_i} e^{\mu_i k_i r(\theta)} r'(\theta)}{c_i(\mu_i, k_i)}\,,
\label{eqn:mixprior}
\end{equation}
where $\omega_i,\, i=1,\cdots, \eta$, are given weights.
\begin{theorem}
		\label{teor:conjmix}
	Suppose that given $\Theta = \theta$, the observable random variables $X_1,\ldots, X_n=:{\bf X}$ are i.i.d.~with common probability function given by \eqref{eq:lexpof}, and that the  
	prior distribution of $\Theta$, $ \pi(\theta) $, is of the form given by \eqref{eqn:mixprior}. Then, the posterior distribution, denoted as $\pi(\theta | {\bf x})$ with ${\bf x} = \{x_1, \ldots, x_n\}$, is of a mixture form as \eqref{eqn:mixprior}:
		\begin{eqnarray}
	\pi(\theta | {\bf x}) &=& \sum_i \omega^*_i \frac{[q(\theta)]^{-k^*_i} e^{\mu^*_i k^*_i r(\theta)} r'(\theta)}{c_i(\mu^*_i, k^*_i)},
	\end{eqnarray}
	where
	\begin{eqnarray}
	\mu^*_i &=& \frac{\mu_i k_i+\sum_j x_j}{k_i+n},\\
	k^*_i &=& k_i+n,\\[0.25cm]
	%
	%
	w^*_i &=& \frac{\omega_i \frac{c_i(\mu^*_i, k^*_i)}{c_i(\mu_i, k_i)}}{\sum_i \omega_i \frac{c_i(\mu^*_i, k^*_i)}{c_i(\mu_i, k_i)}}\,.
	\end{eqnarray}
\end{theorem}
\begin{proof}
	With observations ${\bf X} = {\bf x}$, the posterior distribution is
	\begin{eqnarray*}
	\pi(\theta | {\bf x}) &=& \frac{ \frac{\prod_j p(x_j)e^{r(\theta) \sum_j x_j} }{[q(\theta)]^n} \cdot \sum_{i = 1}^m \omega_i \frac{[q(\theta)]^{-k_i} e^{\mu_i k_i r(\theta)} r'(\theta)}{c_i(\mu_i, k_i)} }{\displaystyle\int_{\Theta} \left(\frac{\prod_j p(x_j)e^{r(\theta) \sum_j x_j} }{[q(\theta)]^n} \cdot \sum_{i = 1}^m \omega_i \frac{[q(\theta)]^{-k_i} e^{\mu_i k_i r(\theta)} r'(\theta)}{c_i(\mu_i, k_i)} \right) \,\rm d\theta} \\[0.25cm]
	%
	%
	&=&\frac{\sum_i \omega_i \frac{[q(\theta)]^{-k_i-n} e^{(\mu_i k_i+\sum_j x_j) r(\theta)} r'(\theta)}{c_i(\mu_i, k_i)}}{\sum_i \omega_i \displaystyle{\int_{\Theta}} \frac{[q(\theta)]^{-k_i-n} e^{(\mu_i k_i+\sum_j x_j) r(\theta)} r'(\theta)}{c_i(\mu_i, k_i)} \,\rm d\theta }
	\\[0.25cm]
	%
	%
	&=&\frac{\sum_i \omega_i \frac{[q(\theta)]^{-k_i-n} e^{(\mu_i k_i+\sum_j x_j) r(\theta)} r'(\theta)}{c_i(\mu_i, k_i)} }{\sum_i \omega_i \frac{c_i(\mu^*_i, k^*_i)}{c_i(\mu_i, k_i)}\displaystyle{\int_\Theta} \frac{[q(\theta)]^{-k^*_i} e^{k^*_i\mu^*_i r(\theta)} r'(\theta)}{c_i^*(\mu^*_i, k^*_i)}\,\rm d\theta }  
	%
	%
	= \sum_i \omega^*_i \frac{[q(\theta)]^{-k^*_i} e^{\mu^*_i k^*_i r(\theta)} r'(\theta)}{c_i(\mu^*_i, k^*_i)}.
	\end{eqnarray*}
\hfill
\end{proof}

From the posterior distribution we can calculate the Bayesian premium, as an estimate for the risk premium [see \cite{klugman2012loss}, Chapters 17 and 18] defined as $\mathbb{E}[X|\theta]=\mu(\theta)$. The Bayesian premium is given by estimate for next rating period, $n+1$, $\mathbb{E}[X_{n+1}|{\bf X} = {\bf x}]$ (with $n$, ${\bf X}$ and ${\bf x}$  as defined in Theorem~\ref{teor:conjmix}) and can be calculated via the posterior distribution $\pi(\theta | {\bf x})$, as
\begin{equation*}
\mathbb{E}[X_{n+1}|{\bf x}]=\mathbb{E}[\mu(\Theta)|{\bf x}]\,.
\end{equation*}
If this expectation is a linear function on the observations then we have exact credibility. This is the case when we work with conjugate distributions.

Moving to our particular concern, the prior distributions under study in our work, both for the claim counts and severity are of the above (mixture) form. As said in Remark~\ref{rem:indep}, we can calculate the premium components, for the claim counts and severity, separately. We are going to consider first the claim frequency and then the severity component.

Before going to that we must write some considerations on premium calculation and estimation. Premia are set to be calculated for rating periods, most commonly the year period. For estimating, risk observations must be organized by periods, suppose that we observed the risk for $m$ years, or rating periods. Then, our observation vector for the claim counts is ${\bf n} = \{n_1,\ldots, n_m\}$. For each year, say year $j$, we'll have $n_j$ individual claims observed, say vector ${\bf y}_j = \{y_{j1},\ldots, y_{jn_j}\}$. This means that observations in year $j$ are set in the two fold vector $(n_j,{\bf y}_j)$. Of course that if we needed to calculate just the annual premium we would just need annual amounts on aggregate, since a premium is set for an aggregate quantity. However, we may consider that claim severities bring some information on an unforseeable stream.  Then, it would be interesting to consider what is the premium contribution for that. Besides, we may also consider that information in both claim counts and severity is the same, then there is no need to double that in premium calculation.

Besides calculating Bayesian premia, we will still need to estimate parameters to feed in the Bayesian estimates (\textit{Empirical Bayes}). This is done in Section~\ref{s:estima}.
\subsection{Frequency component}
If we consider a prior to be a Gamma mixture of the form given by \eqref{eqn:gammamix}, we reach to a posterior with a similar mixture, with updated paramenters. This is given in Lemma~\ref{lem:pstGmix} that follows.
\begin{lemma}\label{lem:pstGmix}
	Let the prior distribution be \eqref{eqn:gammamix}. With a Poisson model distribution, the posterior distribution is still a mixture of Gamma distributions, with updated parameters:
	\begin{eqnarray}
	\pi(\lambda | {\bf n}) &=& w \, \cdot \frac{ (\beta+m)^{\sum_j n_j +\alpha_1}\lambda^{\sum_j n_j +\alpha_1-1}e^{-(m+\beta) \lambda} }{\Gamma(\sum_j n_j+\alpha_1)} \nonumber
	\\[2.5mm]
	&&+ (1-w)\cdot \frac{ (\beta+m)^{\sum_j n_j +\alpha_1+\alpha_2}\lambda^{\sum_j n_j +\alpha_1+\alpha_2-1}e^{-(m+\beta) \lambda} }{\Gamma(\sum_j n_j+\alpha_1+\alpha_2)},
	\label{eqn:posterior}
	\end{eqnarray}
	where
	\begin{eqnarray}
w &=& \frac{1}{1+G(p, \alpha_1, \alpha_2, \beta, {\bf n})} \label{eq:w}\\[1.75mm]
	G(p, \alpha_1, \alpha_2, \beta, {\bf n}) & = & \frac{1-p}{p}\cdot\frac{B(\alpha_1, \alpha_2)}{B(\sum_j n_j +\alpha_1, \alpha_2)}\cdot\left(\frac{\beta}{\beta+m}\right)^{\alpha_2}, \nonumber 
\end{eqnarray}
	 ${\bf n} = \{n_1,\ldots, n_m\}$ is the vector of claim count observations and $m$ is the sample size ($\sum_j \cdot=  \sum_{j=1}^m \cdot $).
\end{lemma}
\begin{proof}
	Under the observations ${\bf n} = \{n_1,\ldots, n_m\}$,
	\begin{eqnarray*}
	\pi(\lambda | {\bf n}) &=& \frac{\prod_j \frac{e^{-\lambda}\lambda^{n_j}}{n_j!} \left\{p\cdot\frac{\lambda^{\alpha_1-1}\beta^{\alpha_1} e^{-\beta \lambda}}{\Gamma(\alpha_1)} + (1-p)\cdot\frac{\lambda^{(\alpha_1+\alpha_2)-1}\beta^{\alpha_1+\alpha_2} e^{-\beta \lambda}}{\Gamma(\alpha_1+\alpha_2)}\right\}}{{\displaystyle\int_\Lambda} \prod_j \frac{e^{-\lambda}\lambda^{n_j}}{n_j!} \left\{p\cdot\frac{\lambda^{\alpha_1-1}\beta^{\alpha_1} e^{-\beta \lambda}}{\Gamma(\alpha_1)} + (1-p)\cdot\frac{\lambda^{(\alpha_1+\alpha_2)-1}\beta^{\alpha_1+\alpha_2} e^{-\beta \lambda}}{\Gamma(\alpha_1+\alpha_2)}\right\}\,\rm d\lambda}\nonumber\\[2.5mm]
		%
	%
	&=& \frac{p\cdot \frac{\lambda^{\sum_j n_j +\alpha_1-1}\beta^{\alpha_1} e^{-(m+\beta) \lambda}}{\Gamma(\alpha_1)} + (1-p) \cdot \frac{\lambda^{\sum_j n_j+(\alpha_1+\alpha_2)-1}\beta^{\alpha_1+\alpha_2} e^{-(m+\beta) \lambda}}{\Gamma(\alpha_1+\alpha_2)}}{p\cdot\frac{\Gamma(\sum_j n_j+\alpha_1)}{\Gamma(\alpha_1)}\left(\frac{\beta}{\beta+m}\right)^{\alpha_1}\left(\frac{1}{\beta+m}\right)^{\sum_j n_j}+ (1-p)\cdot \frac{\Gamma(\sum_j n_j+\alpha_1+\alpha_2)}{\Gamma(\alpha_1+\alpha_2)}\left(\frac{\beta}{\beta+m}\right)^{\alpha_1+\alpha_2}\left(\frac{1}{\beta+m}\right)^{\sum_j n_j} }\nonumber\\[2.5mm]
		%
	%
	&=&\frac{1}{1+G(p, \alpha_1, \alpha_2, \beta, {\bf n})}\cdot \frac{ (\beta+m)^{\sum_j n_j +\alpha_1}\lambda^{\sum_j n_j +\alpha_1-1}e^{-(m+\beta) \lambda} }{\Gamma(\sum_j n_j+\alpha_1)} \nonumber
	\\[2.5mm]
		%
	%
	&& \hspace{0.5cm}  + \frac{G(p, \alpha_1, \alpha_2, \beta, {\bf n})}{1+G(p, \alpha_1, \alpha_2, \beta, {\bf n})}\cdot \frac{ (\beta+m)^{\sum_j n_j +\alpha_1+\alpha_2}\lambda^{\sum_j n_j +\alpha_1+\alpha_2-1}e^{-(m+\beta) \lambda} }{\Gamma(\sum_j n_j+\alpha_1+\alpha_2)}.
	\end{eqnarray*}
	The result is as described. \hfill
\end{proof}

 For a fixed individual claim size, The pure premium estimator will be considered as the posterior mean of the claim frequency component, i.e.,
\begin{equation}
\label{eq:BayesPremium_N}
\mathbb{E}[N_{m+1}| {\bf n}] = \mathbb{E}[\Lambda | {\bf n}]= w\cdot \frac{\sum_j n_j +\alpha_1}{\beta+m} + (1-w)\cdot \frac{\sum_j n_j +\alpha_1+\alpha_2}{m+\beta}\,,
\end{equation}
where $z$ is given by \eqref{eq:w}.

 We note that the Bayesian premium does not have the credibility form as it may look at first [see \cite{klugman2012loss}, Chapters 17 and 18]. Our Bayesian premium estimator is not linear on the observations, the $G(\cdot)$ function depend on vector ${\bf n}$, through a Beta function, and it is not a linear function on the observations. Furthermore, due to Stirling's formula we have that the quocient of the Gamma functions 
 $\Gamma(n+a)/\Gamma(n+b)\sim n^{a-b}$ as $n\rightarrow\infty$, 
 allowing us to write that the function 
 $$
 G(p, \alpha_1, \alpha_2, \beta, {\bf n})\sim \frac{1-p}{p}\frac{\Gamma(\alpha_1)}{\Gamma(\alpha_1+\alpha_2)}
 \left( \frac{\beta \sum_j n_j}{\beta+m}\right)^{\alpha_1} \text{ as } \sum_j n_j \rightarrow \infty\,. 
 $$
 
 In fact, the Bayesian premium is of credibility form when we work with conjugate priors. What we have here is a mix of conjugate priors. We have, though, a mixture of two credibility premia. For instance, we have that the formula 
 \begin{equation}
 \label{eq:BayesPremium_N1}
  \frac{\sum_j n_j +\alpha_1}{\beta+m} =
  \frac{m}{\beta+m}\cdot \left(\frac{\sum_j n_j}{m}\right) +
   \frac{\beta}{\beta+m}\cdot  \left(\frac{\alpha_1}{\beta}\right) 
 \end{equation}
 can be viewed as the credibility premium as if there was only the historical stream claim frequency. Similarly, $({\sum_j n_j +\alpha_1+\alpha_2})/\left({m+\beta}\right)$ gives a corresponding formula when considering the existence of an unforeseeable stream. Then, Formula~\eqref{eq:BayesPremium_N} is a weighted average of these two premia, where the weights depend on the observations.
 
 In order to calculate the Bayesian premium, we still need the claim severity component. A similar methodology can be implemented for the claim severity component.
\subsection{Severity component}
For the severity component, following what we wrote on the mean of the exponetial distribution at the end of Section~\ref{s:sev_comp}, we  can set the prior as
\begin{equation}
\pi(\theta) = \nu\Delta_\theta(\{\mu\}) + (1-\nu) \frac{\theta^{\delta-1} \sigma^{\delta} e^{-\sigma \theta}}{\Gamma(\delta)}\,, 
\label{eqn:sizemix}
\end{equation}
where $\Delta_\theta(\{\mu\}) = \mathbf{1}_{\{\mu\}}(\theta)$ is the Dirac measure, and $ \nu $ is as given in \eqref{eqn:clmdist}. The posterior is still a mixture, as follows. As said, given $ \Theta=\theta $, the conditional  distribution of a single severity is exponential with mean $1/\theta$.

We are calculating a premium estimate for the severity component only, based on the observed single quantities. Let's consider that the sample is generically of  size $m^\ast$ and the observation vector is ${\bf y} = \{y_1,\ldots, y_{m{^\ast}}\}$. If sample size of claim counts ${\bf N}={\bf n}$ is $m$ then $m^\ast=\sum_{j=1}^{m}n_j\,$, $n_j\geq 1$. The posterior distribution is set in the following Lemma~\ref{lem:sevpost} and the premium follows after in Expression~\eqref{eq:BayesPremium_Y}. 
\begin{lemma}\label{lem:sevpost}
	Let the prior distribution be \eqref{eqn:sizemix} and the conditional  distribution of a single severity, given $\Theta=\theta$, be exponential with mean $1/\theta$. The posterior distribution is still a mixture distribution in terms of the prior form with updated parameters, such that
	\begin{equation}
	\pi(\theta | {\bf y}) = \omega\cdot \Delta_\theta(\{\mu\}) + (1-\omega) \cdot \frac{ (\sigma+\sum_i y_i)^{{m^\ast}+\delta}\, \theta^{{m^\ast} +\delta-1} \, e^{-(\sigma+\sum_i y_i) \theta} }{\Gamma({m^\ast}+\delta)},
	\label{eqn:sizeposterior}
	\end{equation}
	where
	\begin{eqnarray}
\omega &=&	\frac{1}{1+\varphi(\nu, \mu, \delta, \sigma, {\bf y})} \label{eq:omega} \\[0.15cm]
	\varphi(\nu, \mu, \delta, \sigma, {\bf y}) & = & \frac{1-\nu}{\nu}\cdot\frac{\Gamma(m^\ast+\delta) \sigma^{\delta}}{\Gamma(\delta) (\sigma+\sum_i y_i)^{{m^\ast}+\delta}}\cdot \mu^{-{m^\ast}} e^{\mu \sum_i y_i}\,,\nonumber
	\end {eqnarray}
	with $ \nu $ as given in \eqref{eqn:clmdist},  ${\bf y} = \{y_1,\ldots, y_{m{^\ast}}\}$ and $\sum_i \cdot =\sum_{i=1}^{m^\ast} \cdot$.
\end{lemma}
\begin{proof}
	Under the observations ${\bf y} = \{y_1,\ldots, y_{m^\ast}\}$, and let $\Pi_\Theta(\cdot)$ be the distribution function of $\Theta$
	\begin{eqnarray}
	\pi(\theta| {\bf n}) &=& \frac{\left(\prod_{i=1}^{m^\ast} \mu e^{-\mu y_i}\right) 
		\left\{\nu\Delta_\theta(\{\mu\})\right\}
		+
		\left(\prod_{i=1}^{m^\ast} \theta e^{-\theta y_i}\right)
		 \left\{ (1-\nu) \frac{\theta^{\delta-1} \sigma^{\delta} e^{-\sigma \theta}}{\Gamma(\delta)}\right\}}
	 {\int_\Theta \left(\prod_{i=1}^{m^\ast}  \theta e^{-\theta y_i}\right) \rm d\Pi_\Theta(\theta) }
	 	\nonumber\\
	&& \nonumber\\[0.1cm]
	&=& 
	\frac{\nu \left(\mu^{m^\ast} e^{-\mu \sum_i y_i}\right) \Delta_\theta(\{\mu\}) + (1-\nu) \, \frac{\theta^{{m^\ast}+\delta -1}\sigma^{\delta} e^{-(\sigma+\sum_i y_i) \theta}}{\Gamma({m^\ast}+\delta)}}
	{\nu \, \left(\mu^{m^\ast} e^{-\mu \sum_i y_i}\right) + (1-\nu) \, \frac{\Gamma({m^\ast}+\delta) \sigma^{{m^\ast}+\delta}}{\Gamma(\delta) (\sigma+\sum_i y_i)^{{m^\ast}+\delta}} }\nonumber\\
	&& \nonumber\\[0.1cm]
	&=&
	\frac{1}{1+\varphi(\nu, \mu, \delta, \sigma, {\bf y})} \, \Delta_\theta(\{\mu\}) + \frac{\varphi(\nu, \mu, \delta, \sigma, {\bf y})}{1+\varphi(\nu, \mu, \delta, \sigma, {\bf y})}\, \frac{ (\sigma+\sum_i y_i)^{{m^\ast}+\delta} \, \theta^{{m^\ast} +\delta-1} \, e^{-(\sigma+\sum_i y_i) \theta} }{\Gamma({m^\ast}+\delta)} \,. \nonumber \\
	& & 
	\end{eqnarray}
	The result is as described. \hfill
\end{proof}

Unlike \eqref{eq:BayesPremium_N}, the Bayesian premium for the severity component is not given directly by the posterior mean.  That is given by the posterior expectation  $\mathbb{E}[\mu(\Theta)|{\bf x}]$, and is the weighted average between the mean of the exponential ($ 1/\mu $) and the mean of the (posterior) $Gamma({m^\ast}+\delta, \sigma+\sum_i y_i)$:
\begin{equation}
\label{eq:BayesPremium_Y} 
\mathbb{E}[Y_{n+1}| {\bf y}] = \mathbb{E}[\mu(\Theta)|{\bf x}] = \omega \, \frac{1}{\mu} + (1-\omega)\,  \frac{{m^\ast} +\delta}{\sum_{i=1}^{m^\ast} y_i+\sigma}\,,
\end{equation}
with $\omega$ calculated through~\eqref{eq:omega}.

{As a remark, we can see that as $m^*\rightarrow 0$, $\sum_i y_i \rightarrow 0$, and 
	\[
	\varphi(\nu, \mu, \delta, \sigma, {\bf y}) \sim \frac{1-\nu}{\nu}.
	\]}

Hence, the Bayesian premium, estimate for the next rating period $m+1$ and denoted as $P_{m+1}$ can be given by multiplying \eqref{eq:BayesPremium_N} with \eqref{eq:BayesPremium_Y}, feeding \eqref{eq:w} and \eqref{eq:omega},
\begin{eqnarray}
	\mathbb{E}[P_{m+1}| {\bf n} ,{\bf y}] &=& \left( w \cdot \frac{\sum_{j=1}^{m} n_j +\alpha_1}{\beta+m} + (1-w)\cdot \frac{\sum_{j=1}^{m} n_j +\alpha_1+\alpha_2}{m+\beta}\right) \nonumber
	\\ 
	&&  \hspace{2.85cm} 
	\times \left(\omega \cdot \frac{1}{\mu} + (1-\omega) \cdot \frac{m^\ast +\delta}{\sum_{j=1}^{m^\ast} y_i+\sigma}\right). \label{eqn:prem}
\end{eqnarray}
In this formula we have a set of prior parameters that are most often unkown, and they need to be estimated from observed data, leading to the \textit{Empirical Bayesian Premium}. This is not an easy task, commonly when we have to deal with mixtures of distributions, we work with an unusual quantity of estimating parameters. 

In the coming subsection we present a numerical example to address the asymptotics of functions $G$ and  $\varphi$, from \eqref{eq:w} and \eqref{eq:omega}, in the derived formulae, and to understand the non-linearity of these functions with the observed claims (counts and costs).
\subsection{Numerical example}
%
{For illustration, we work with a numerical example setting a scenario where the claim counts follow a $(0, 2, 1)$ pattern in every three consecutive years over a 12-year period. That is to say, we investigate one policyholder's claim histories and his/her claim counts for each consecutive three years are 0, 2, 1, respectively.\\ 

Here, each individual claim size is simulated to be the prior mean
	\begin{equation*}
	\mathbb{E}[Y] = \nu\frac{1}{\mu} + (1-\nu)\frac{\delta}{\sigma},
	\end{equation*}
plus a 
{Normal distributed random error}
 with mean $0$ and standard deviation $0.1$.
	Further, we assume that $\alpha_1 = 3$, $\alpha_2 = 1$, $\beta = 0.5$, $p = 0.6$, $\mu = 1$, $\delta = 2$, $\sigma = 0.5$. As a consequence, applying Formula~\eqref{eqn:prem}},  we compute annual Bayesian premia, figures and graph are shown in Table~\ref{tab:prem_evol} and Figure~\ref{fig:prem_evol}, respectively.
	
\begin{table}
	\centering
	\begin{tabular}{|c|c|c|c|}
		\hline
		Period & $\sum n_j$ & $\sum_{i=1}^{m^*} y_i$ & Premium\\
		\hline
		0 & 0 & 0 & 0.8840\\\hline
		1 & 0 & 0 & 0.6049\\\hline           
		2 & 2 & 2.52 & 0.8135\\\hline
		3 & 3 & 3.98 & 0.8445\\\hline
		4 & 3 & 3.98 & 0.7473\\\hline
		5 & 5 & 6.60 & 0.8361\\\hline
		6 & 6 & 8.07 & 0.8047\\\hline
		7 & 6 & 8.07 & 0.7360\\\hline
		8 & 8 & 10.59 & 0.7968\\\hline      
		9 & 9 & 11.82 & 0.7965\\\hline
		10 & 9 & 11.82 & 0.7423\\\hline
		11 & 11 & 14.50 & 0.7852\\\hline
		12 & 12 & 15.83 & 0.7813\\\hline
	\end{tabular}
\caption{Annual Bayesian premia with $(0, 2, 1)$ pattern}
\label{tab:prem_evol}
\end{table}
%
%
%
\begin{figure}
\begin{center}
	\includegraphics[scale=0.3]{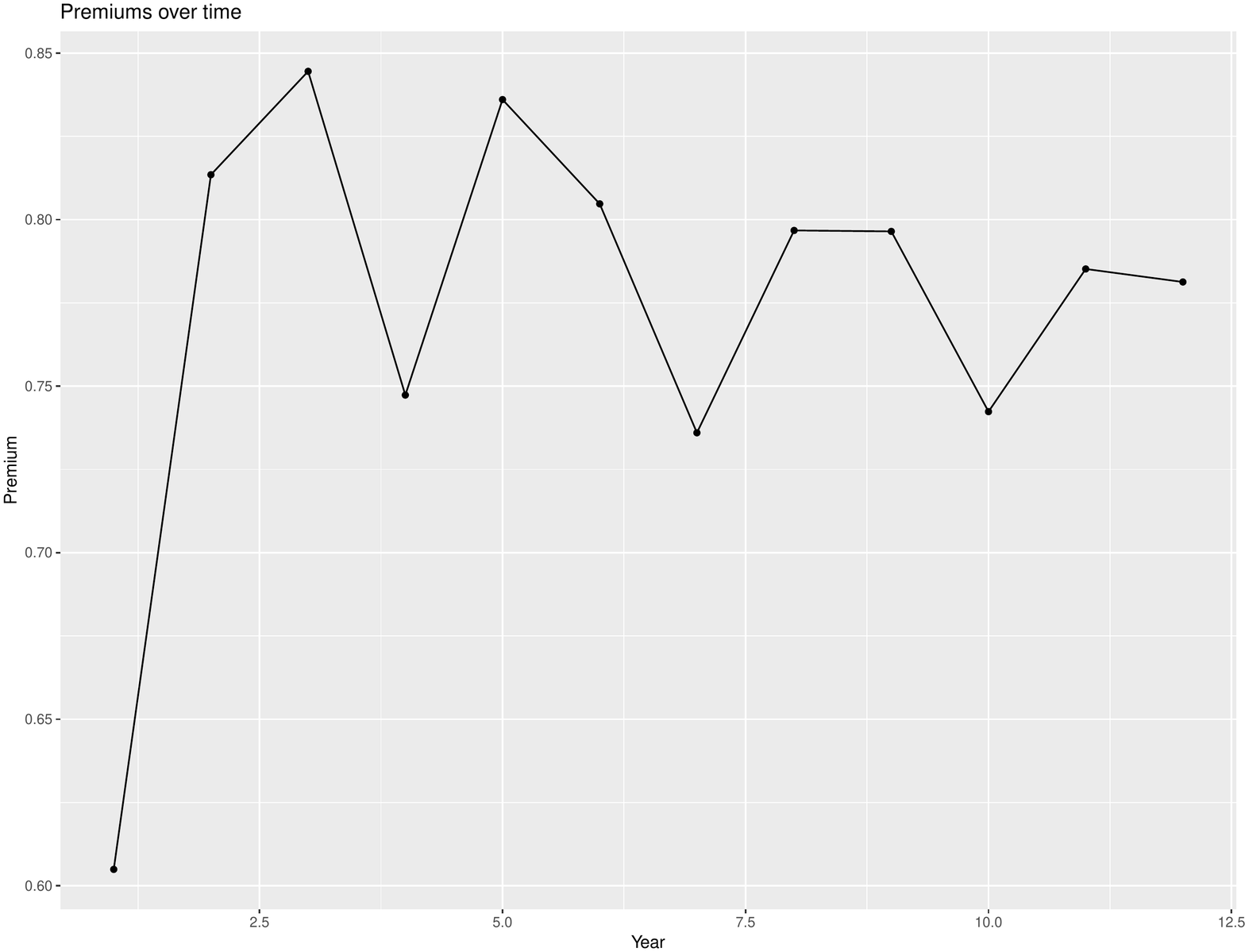}
	\includegraphics[scale=0.3]{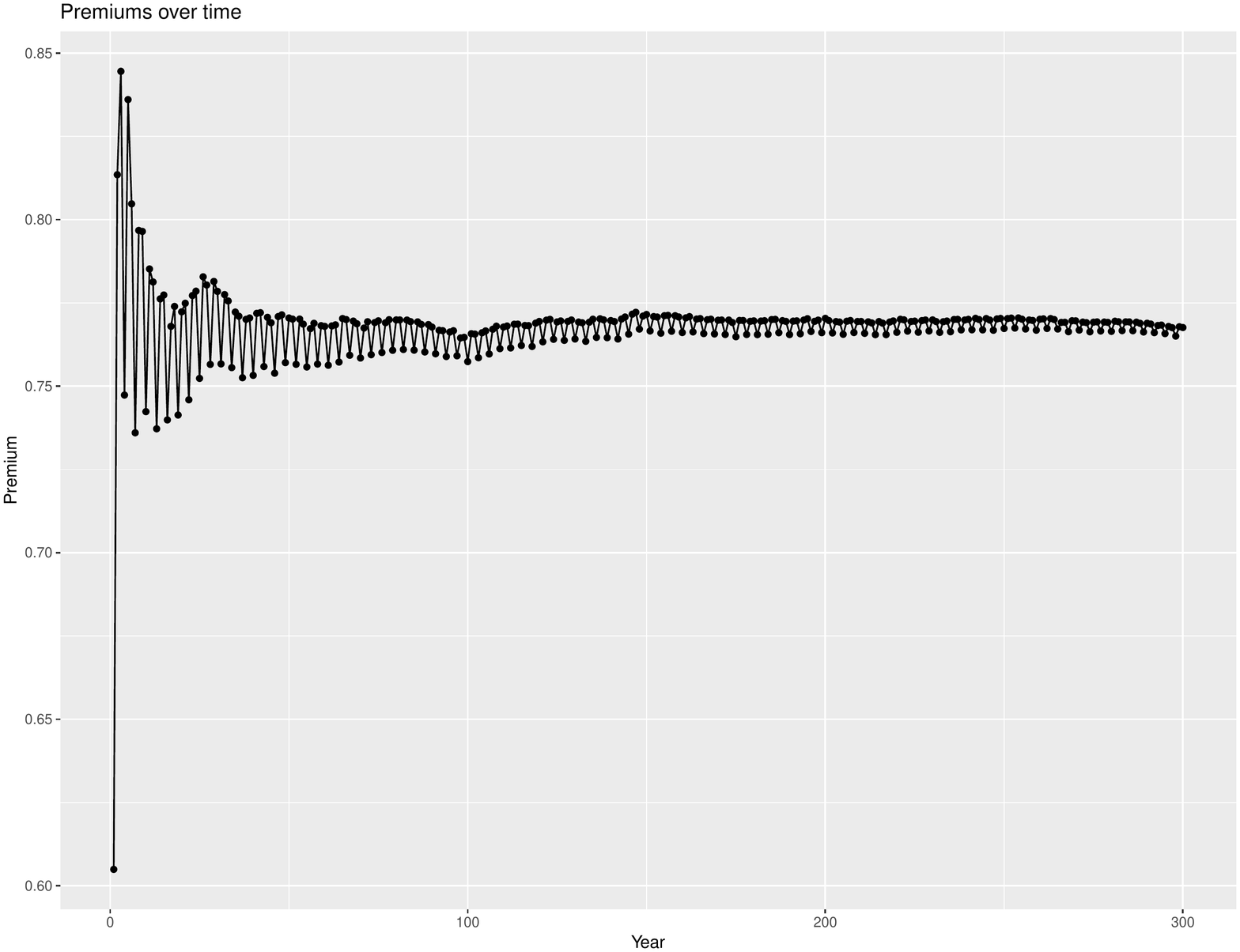}
\end{center}
\caption{Premium evolution with $(0, 2, 1)$ pattern}
\label{fig:prem_evol}
\end{figure}
%
%
{The pattern of the premia reflects the periodical pattern we assigned for the claim counts over years. Note that the no claim situation in a particular year leads to a significant decrease in the premium for the upcoming year, yet the amount of reduction diminishes over time. In addition, the rate of increase in the annual premium is generally higher with two claims when compared to that with one claim only, as expected. Looking at the graph on the right hand side, in the long run, the premium converges to a fixed value with the growth of cumulative claim counts and costs.
}
	

{The premium function depends on two quantities which are functions of past claim records. They are function $G(\cdot)$ and $\varphi(\cdot)$ from \eqref{eq:w} and \eqref{eq:omega}, respectively. Now, we turn to look at some numerical results for them and the corresponding weight functions they generate, respectively.}
%
%
%
{We first show how $G(\cdot)$ varies according to years elapsed. It can be seen that the values of $G(\cdot)$ also behaves in a periodical manner due to the scenario setting, and eventually the speed of growth slows down. We refer to Figure~\ref{fig:G_asym} and on the left hand side graph.
Graph on the right hand side is the weight function governed by $G(\cdot)$. That shows the weights in the mixture posterior distribution for the claim frequency component. A large $G(\cdot)$ value implies a smaller weight assigned for the historical risk stream. Both graphs demonstrate that $G(\cdot)$ as well as the historical weight will converge to a fixed value ultimately with extensive fluctuations in earlier years. One should be aware that over time cumulative claim counts and costs are non-decreasing. The graphs also verify the asymptotics of $G(\cdot)$ and the weight function when accumulated claim counts $m^* = \sum_j n_j \rightarrow\infty$, as mentioned earlier.}
\begin{figure}
\begin{center}
	\includegraphics[scale=0.3]{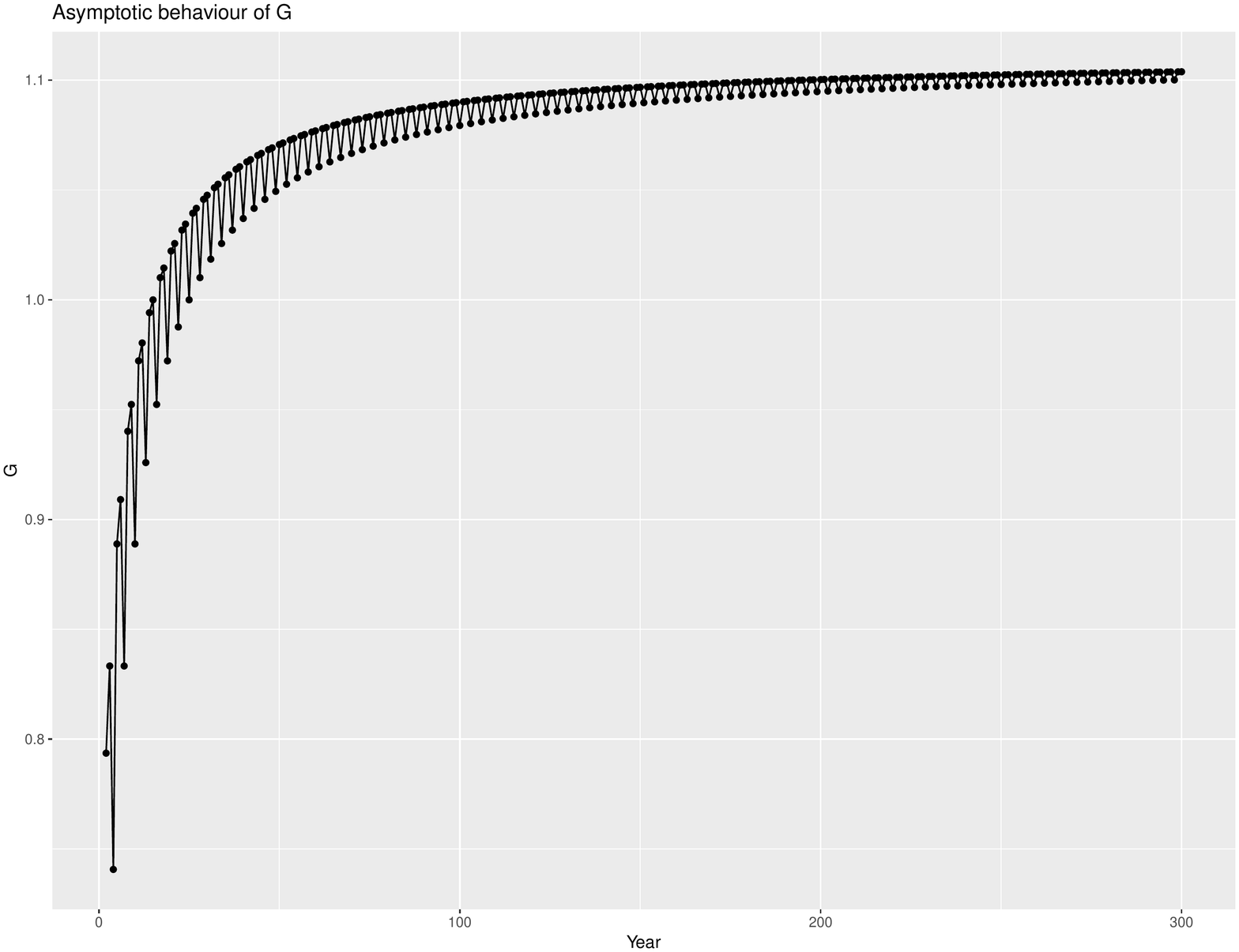}
	\includegraphics[scale=0.3]{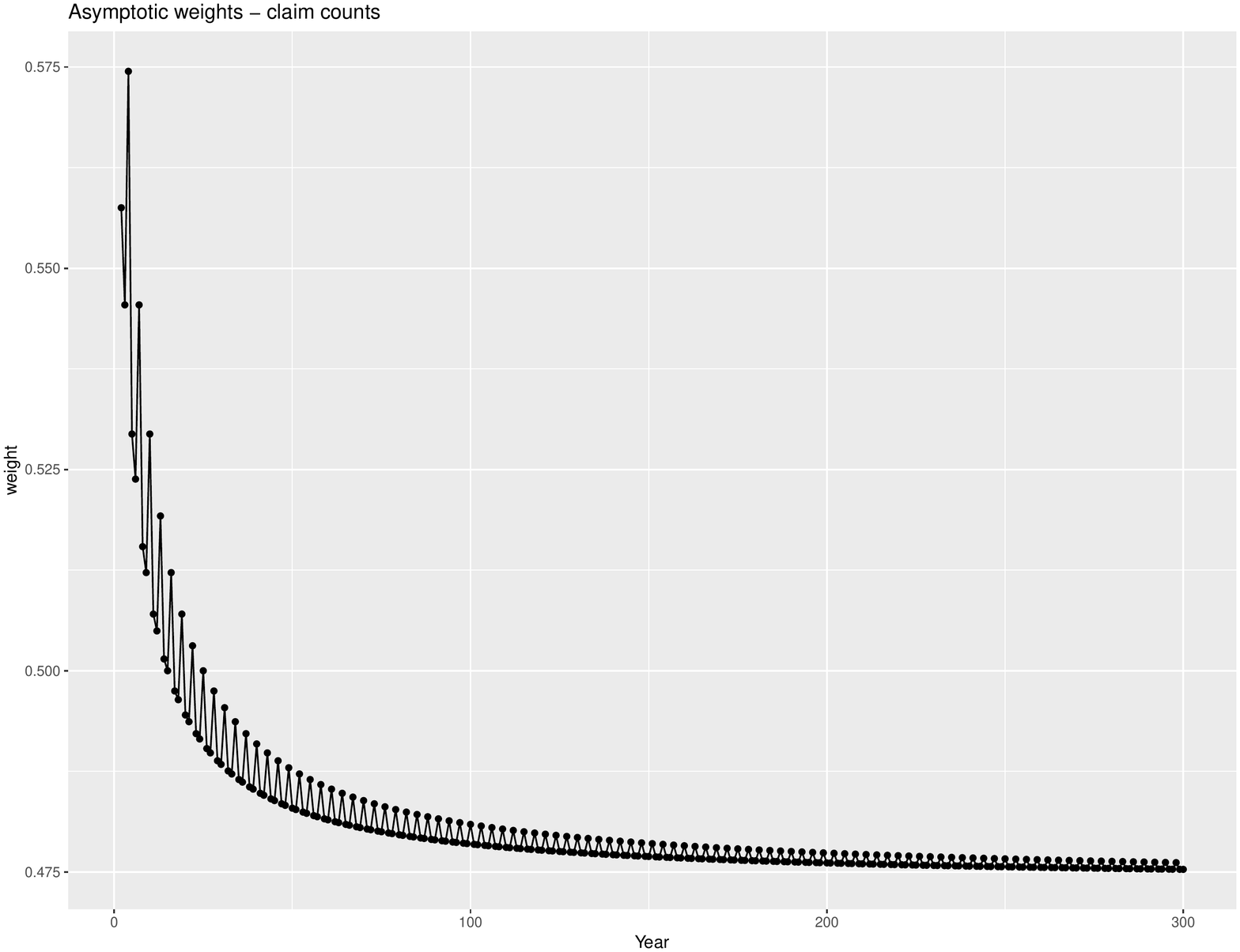}
\end{center}
\caption{Asymptotic behaviour of function $G$ and the posterior weight generated by $G$}
\label{fig:G_asym}
\end{figure}

Next we address the $\varphi$ and its associated weight function. Similarly, we can justify the asymptotic result of $\varphi(\cdot)$ when $\sum_i y_i \rightarrow 0$ is given by a fixed value dependent on $\nu$. Additionally, $\varphi(\cdot)$ has a different behaviour compared to $G(\cdot)$ overtime. See Figure~\ref{fig:phi_asym_claims}. Firstly, $G(\cdot)$ and $\varphi(\cdot)$ differ in convexity, which leads to opposite asymptomatic directions. Secondly, when $\varphi(\cdot)$ gets larger, the variations tend to be slightly larger as well. 
Graph on the right shows the weight function for posterior distribution governed by $\varphi(\cdot)$. When the weight is larger, claim sizes are dominated by the exponential part in the mixture. And a smaller weight implies dominance by the gamma part. We could see that the weight starts at values closer to 1 and gradually lean towards values closer to 0. That means as time goes by, more weights are being assigned to the unforseeable risk stream when calculating premiums.
\begin{figure}
\begin{center}
	\includegraphics[scale=0.3]{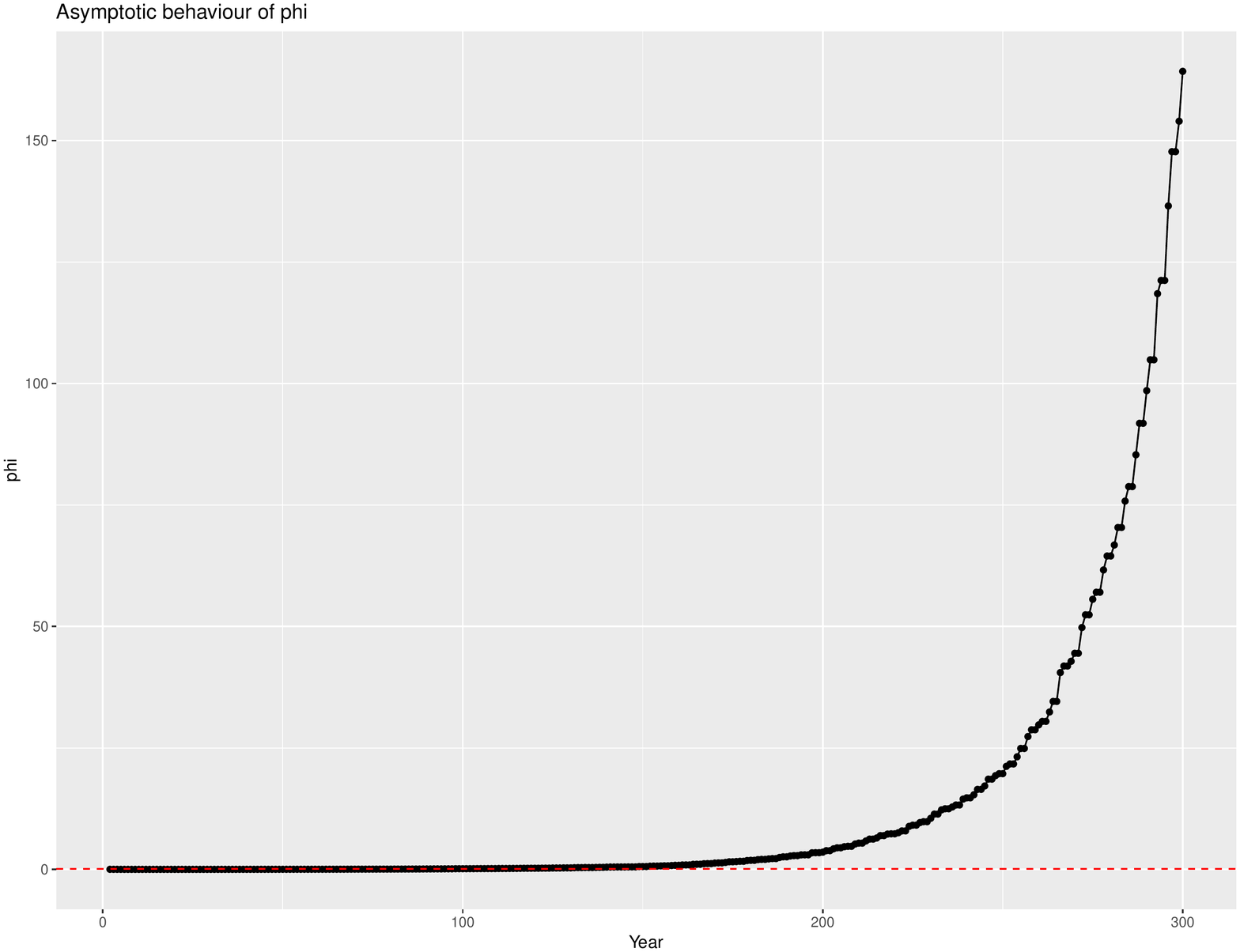}
	\includegraphics[scale=0.3]{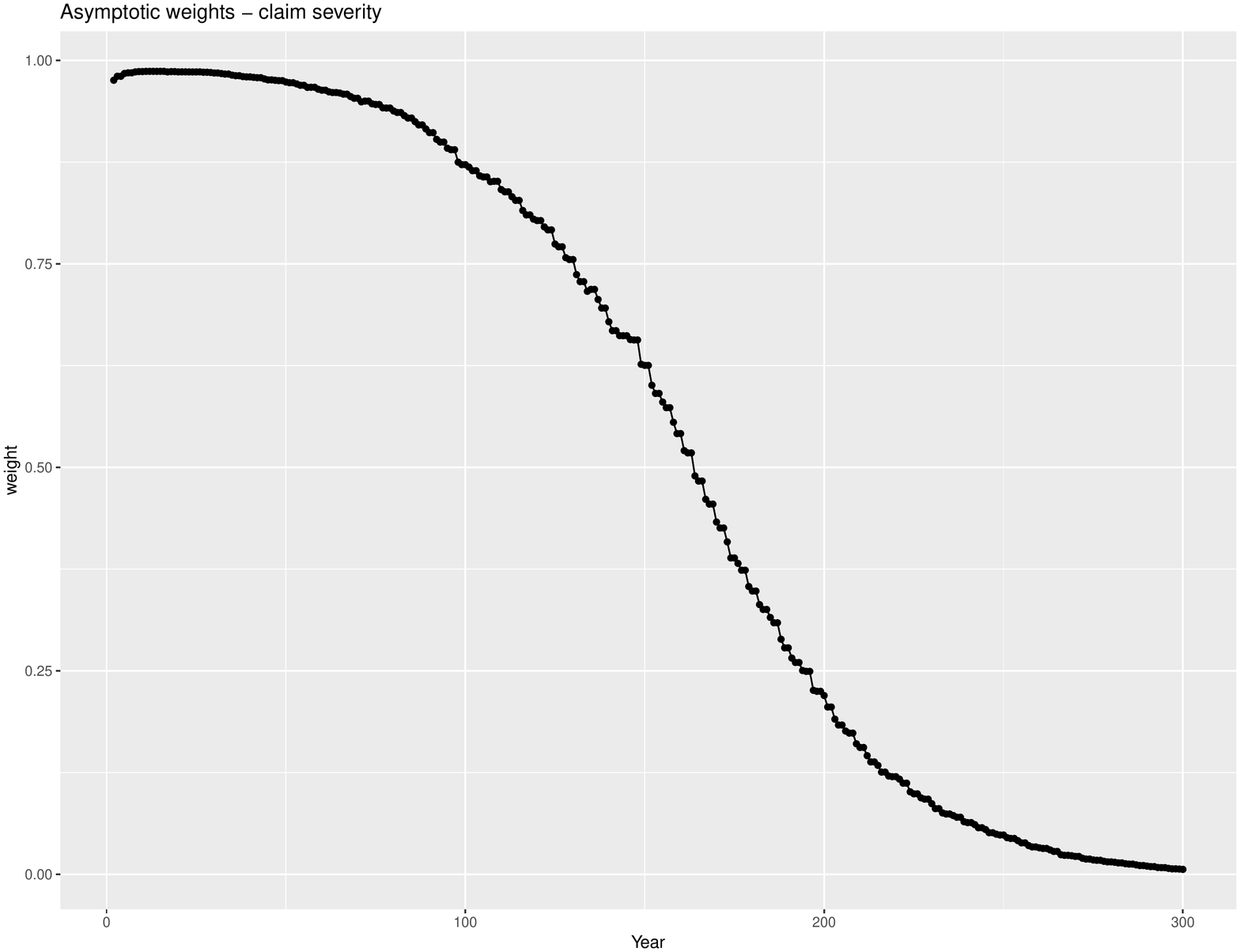}
\end{center}
\caption{Asymptotic behaviour of function $\varphi$ and the posterior weight generated by $\varphi$}
\label{fig:phi_asym_claims}
\end{figure}

\section{Parameter estimation and EM algorithm}
\label{s:estima}
We have seven prior parameters to estimate, the direct use of the Maximum Likelihood Estimation (briefly MLE) method is not workable. Looking a the literature we found that the use the Expectation-Maximization (EM) algorithm suited our purpose, it is said to find also, say, \textit{best fit} estimates. It is an iterative way to approximate Maximum Likelyhood Estimates (MLE), it uses a multi-step process. It is said to be a method to find MLE for model parameters in the presence of  incomplete data [see \cite{dempster1977maximum}], has missing data points, or when existing unobserved (hidden) latent variables. See also \cite{couvreur1997algorithm}.

It works by choosing random values, guesses, for the missing data points,  then uses those guesses to estimate a second set of data. 
The new values are used to create a better guess for the first set, and the process continues until the algorithm converges to a fixed point.

 In this section we explain briefly how we employed the EM algorithm to estimate the prior parameters of our model based on a data set from a Portuguese insurer. Unfortunately we could not find a data set having both the claim counts and the corresponding severities. Often data is recorded on aggregate. The implementation of the algorithm depend on the distribution choice. We assumed the use of a simple distribution as a choice for the severity, other realistic choices may bring other problems to the estimation procedure, like numerical ones. However, some other choices are doable likewise.
\subsection{Data description}
We were provided with quarterly claim count data from a motor insurance portfolio. Namely, it records the total number of claims arising from the portfolio every quarter and there were a total of approximately 180 quarters recorded. Therefore, we could treat each entry as an observation for $N(t)$ for a fixed $t$ which stands for a quarter here. We start to write $N$ for short in the sequel since $N(t)$ is a random variable for fixed $t$. We took care of the Third Party Liability claims only as those are the obligatory components for a car insurance policy. It was assumed that the portfolio is closed over the underlying period and that claim counts for each quarter are independent observations for $N$.

Table~\ref{t:ctstats} shows a brief summary of the data under consideration. The last two columns show the standard deviation and the coefficient of variation for this dataset. Its distribution can be also visualised in the histogram shown by Figure~\ref{fig:hist1}. Clearly, there is a separation around 6000, which serves as a clue for the adoption of a mixture model as we theoretically derived above. We will explain next in more details about the estimation procedures.
\begin{table}[b]
	\center
	\begin{tabular}{cccccccc}
		\hline\hline
		{\bf Min.} & {\bf 1st Qu.} & {\bf Median} & {\bf Mean} & {\bf 3rd Qu.} & {\bf Max.} & {\bf S.D.} & {\bf C.V.}\\
		3685 & 4879 & 5373 & 5552 & 6432 & 7316 & 902.844 & 0.1626\\
		\hline
	\end{tabular}
	\caption{Summary of statistics}
	\label{t:ctstats}
\end{table}
\begin{figure}
	\center
	\includegraphics[scale=0.3]{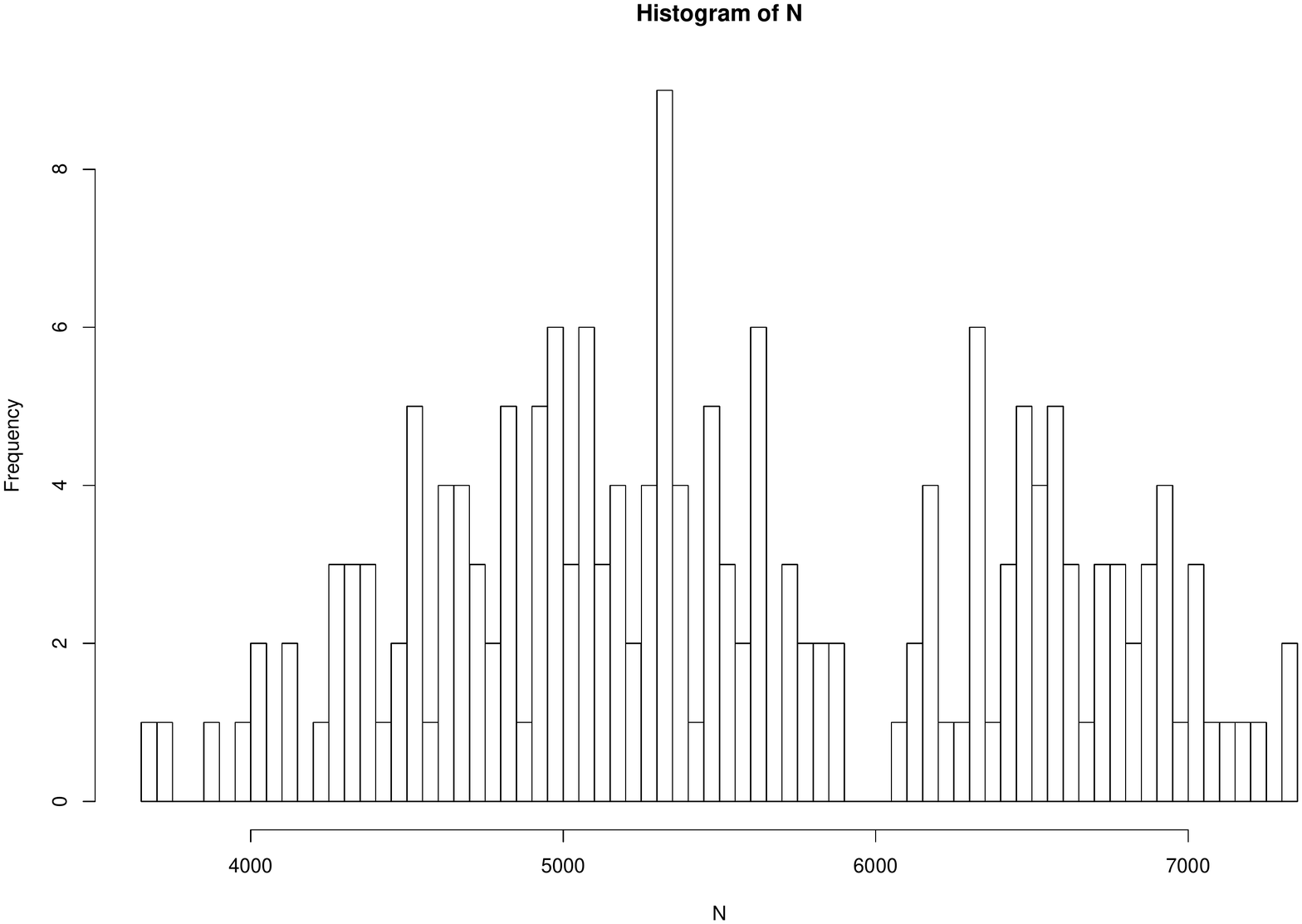}
	\caption{Histogram of claim counts}
	\label{fig:hist1}
\end{figure}
\subsection{EM algorithm}
\subsubsection{Claim frequency component}
\label{sec:emfreq}
Recall that we work with a mixture model and that we only observe a global $N_i \in \mathbb{N}$, we mean, we do not know which negative binomial distribution it comes from. Here, we introduce a latent variable, denoted as $Z$, representing the missing information of which distribution it takes for each observation. Since we only have  two mixed distributions, $Z\frown Bernoulli(p)$ is a Bernoulli random variable with $\mathbb{P}\{Z = 1\} = q=1-\mathbb{P}\{Z = 0\}$. Now, the complete information is given by vector $\{N, Z\}$. We can write the complete likelihood function as follows.
\begin{eqnarray*}
L(\theta|N, Z)	&=& \prod_{i=1}^{m} \left[p\binom{n_i+\alpha_1-1}{n_i}\left(\frac{\beta}{\beta +1}\right)^{\alpha_1}\left(\frac{1}{\beta + 1}\right)^{n_i} \right]^{z_i}\\
	& & \hspace{0.5cm}
	 \times\left[(1-p)\binom{n_i+\alpha_1+\alpha_2-1}{n_i}\left(\frac{\beta}{\beta + 1}\right)^{\alpha_1+\alpha_2}\left(\frac{1}{\beta + 1}\right)^{n_i} \right]^{1-z_i}\, ,
\end{eqnarray*}
and according to the law of total probability we have, 
$$\mathbb{P}\{N\} = \mathbb{P}\{N | Z = 1\}\mathbb{P}\{Z = 1\} + \mathbb{P}\{N | Z = 0\}\mathbb{P}\{Z = 0\}.$$
Here, $\theta = \{\alpha_1, \alpha_2, \beta, p\}$ are the parameters we are interested in estimating for the claim count distribution. Correspondingly, the log-likelihood function is given by, from above,
\begin{eqnarray}
\mathcal{L}(\theta|N, Z) &=& \sum_i z_i \log\left[p\binom{n_i+\alpha_1-1}{n_i}\left(\frac{\beta}{\beta + 1}\right)^{\alpha_1}\left(\frac{1}{\beta + 1}\right)^{n_i} \right]\nonumber\\
&& \hspace{0.35cm}
+\sum_i (1-z_i)\log \left[(1-p)\binom{n_i+\alpha_1+\alpha_2-1}{n_i}\left(\frac{\beta}{\beta + 1}\right)^{\alpha_1+\alpha_2}\left(\frac{1}{\beta + 1}\right)^{n_i} \right].
\label{eqn:loglikelihood}
\end{eqnarray}
%

We summarize how the algorithm works. It is an iterative process where each iteration consists of two steps, the E-step and the M-step, standing for Expectation and Maximisation, respectively.
\begin{enumerate}
	\item We begin with an initially determined parameter values $\theta^{(0)} = \{\alpha_1^{(0)}, \alpha_2^{(0)}, \beta^{(0)}, p^{(0)}\}$.
	\item E-Step 
	
	For the $(l+1)$-{th} iteration, $l = 0, 1, \dots$ , we first seek for the expected value of $Z_i$ conditional on the observations together with the current parameter estimates $\theta^{(l)}$, i.e., estimates from the previous $l$-{th} iteration. This is denoted as $ \tau_i $,
	\begin{eqnarray*}
		\tau_i &=& \mathbb{E} [Z_i| N,\theta^{(l)}] = 1 \times \mathbb{P}\{Z_i = 1|N_i,\theta^{(l)}\} + 0 \times \mathbb{P}\{Z_i = 0|N_i,\theta^{(l)}\} = \frac{\mathbb{P}\{Z_i = 1, N_i=n_i|\theta^{(l)}\}}{\mathbb{P}\{N_i| \theta^{(l)}\}}\\[0.15cm]
		&=&\frac{p^{(l)}\binom{n_i+\alpha^{(l)}_1-1}{n_i}\left(\frac{\beta^{(l)}}{\beta^{(l)} + 1}\right)^{\alpha^{(l)}_1}\left(\frac{1}{\beta^{(l)} + 1}\right)^{n_i}}{p^{(l)}\binom{n_i+\alpha^{(l)}_1-1}{n_i}\left(\frac{\beta^{(l)}}{\beta^{(l)} + 1}\right)^{\alpha^{(l)}_1}\left(\frac{1}{\beta^{(l)} + 1}\right)^{n_i} + (1-p^{(l)})\binom{n_i+\alpha^{(l)}_1+\alpha^{(l)}_2-1}{n_i}\left(\frac{\beta^{(l)}}{\beta^{(l)} + 1}\right)^{\alpha^{(l)}_1+\alpha^{(l)}_2}\left(\frac{1}{\beta^{(l)} + 1}\right)^{n_i}}
	\end{eqnarray*}
	
	Subsequently, based on  Expression~\eqref{eqn:loglikelihood}, we compute the expectation the log-likelihood function with respect to the conditional distribution of $Z$, given $N$ under the current estimates $\theta^{(l)}$, denoted as $ Q(\theta |\theta^{(l)}) $. Considering we have we have $m$ independent observations, we have
	\small{\begin{eqnarray}
		&& Q(\theta |\theta^{(l)}) = \mathbb{E}_{Z_i | N_i, \theta^{(l)}}[\mathcal{L}(\theta | N_i,Z_i)]\label{eqn:expectation}\\[0.15cm]
		&=&\sum_{i=1}^{m}\tau_i \left[\log p^{(l)} + \log\binom{n_i+\alpha_1^{(l)}-1}{n_i}\ + \alpha_1^{(l)}\log \frac{\beta^{(l)}}{\beta^{(l)} + 1} + n_i \log \frac{1}{\beta^{(l)} + 1}\right] \nonumber\\
		&& + \sum_{i=1}^{m} (1-\tau_i) \left[\log (1- p^{(l)}) + \log\binom{n_i+\alpha_1^{(l)} + \alpha_2^{(l)} -1}{n_i} + (\alpha_1^{(l)}+\alpha_2^{(l)})\log \frac{\beta^{(l)}}{\beta^{(l)} + 1} + n_i \log \frac{1}{\beta^{(l)} + 1}\right]\nonumber\\[0.15cm]
		&=& \log p^{(l)} \sum_{i}\tau_i + \sum_i\tau_i \log\binom{n_i+\alpha_1^{(l)} -1}{n_i} + \alpha_1^{(l)} m\log \frac{\beta^{(l)}}{\beta^{(l)} + 1}+ \log \frac{1}{\beta^{(l)} + 1}\sum_i n_i\nonumber\\
		&& + \log (1- p^{(l)}) \sum_{i} (1-\tau_i) + \sum_{i} (1-\tau_i)\log\binom{n_i+\alpha_1^{(l)} + \alpha_2^{(l)} -1}{n_i} +\alpha_2^{(l)}\log \frac{\beta^{(l)}}{\beta^{(l)} + 1}\sum_{i} (1-\tau_i).\nonumber
		\end{eqnarray}}
	\item M-Step\\
	The maximisation step is set to find the parameter values that maximises function  \eqref{eqn:expectation} and they become the estimates to be used in the next iteration, i.e.,
	\begin{equation}
	\theta^{(l+1)} =\operatornamewithlimits{argmax}_{\theta} Q(\theta |\theta^{(l)})
	\label{eqn:max}
	\end{equation}
	For this, we take the gradient of $Q(\theta |\theta^{(l)})$, equate to zero and solve for $\{\alpha_1, \alpha_2, \beta, p\}$ simultaneously.
	\begin{eqnarray}
	\frac{\partial Q}{\partial \alpha_1} &=& \sum_i \tau_i [\digamma(n_i+\alpha_1) - \digamma(\alpha_1)] + \sum_i(1-\tau_i)[\digamma(n_i+\alpha_1+\alpha_2) - \digamma(\alpha_1 + \alpha_2)] + m\log\frac{\beta}{1+\beta}=0;\nonumber\\
	\frac{\partial Q}{\partial \alpha_2} &=& \sum_i(1-\tau_i)\left[\digamma(n_i+\alpha_1+\alpha_2) - \digamma(\alpha_1 + \alpha_2)] + \log\frac{\beta}{1+\beta}\right]=0;\nonumber\\
	\frac{\partial Q}{\partial \beta} &=& \frac{\alpha_1 m + \alpha_2\sum_i(1-\tau_i) - \beta\sum_i n_i}{\beta(1+\beta)}=0;\nonumber\\
	\frac{\partial Q}{\partial p} &=& \frac{\sum_i\tau_i}{p} - \frac{\sum_i(1-\tau_i)}{1-p}=0,\label{eqn:partialp}
	\end{eqnarray}
	where $\digamma(\cdot)$ is the digamma function denoting the logarithmic derivative of a gamma function, i.e.,
	\[
	\digamma(x) = \frac{d}{d x}\log(\Gamma(x)) = \frac{\Gamma'(x)}{\Gamma(x)}.
	\]
	The estimates for the subsequent iteration $(l+1)$-{th} will be the solutions of the above equations. Equation \eqref{eqn:partialp} does not depend on others than parameter $p$, it can be solved directly with explicit representation
	\[
	p^{(l+1)} = \frac{\sum_i \tau_i}{m}.
	\]
	For the other three parameters, we can only solve numerically. The {\it nleqslv} package in R was employed here, it  solves systems of non-linear equations. As a consequence, we can obtain estimated parameters at this iteration $\theta^{(l+1)} = \{\alpha_1^{(l+1)}, \alpha_2^{(l+1)}, \beta^{(l+1)}, p^{(l+1)}\}$.
	\item Plug $\theta^{(l+1)}$ into the $(l+2)$-{th} iteration and repeat the above steps until convergence.
\end{enumerate}

Initiating values using the method of moments yields $\alpha_1^{(0)} = 96.14042;\, \alpha_2^{(0)} = 31.54888; \, \beta^{(0)} = 0.01927362;\,  p^{(0)}= 0.6555556$. Then we implemented EM algorithm on the chosen dataset mentioned earlier. At the tolerance level of $0.001$, we reached convergence with 75 iterations and the resulting estimates for $\alpha_1, \alpha_2, \beta, p$ are
\[
\alpha_1 = 97.55820446;\, \alpha_2 = 30.14706672;\, \beta = 0.01978072;\,  p = 0.5929959.
\]
Based on these parameters, we randomly generated a sequence of numbers and compared them with the observed data. We put together both histograms, in different colors, in Figure \ref{fig:hist2}.  Visually it appears to be a good fit.
\begin{figure}
	\center
	\includegraphics[scale=0.3]{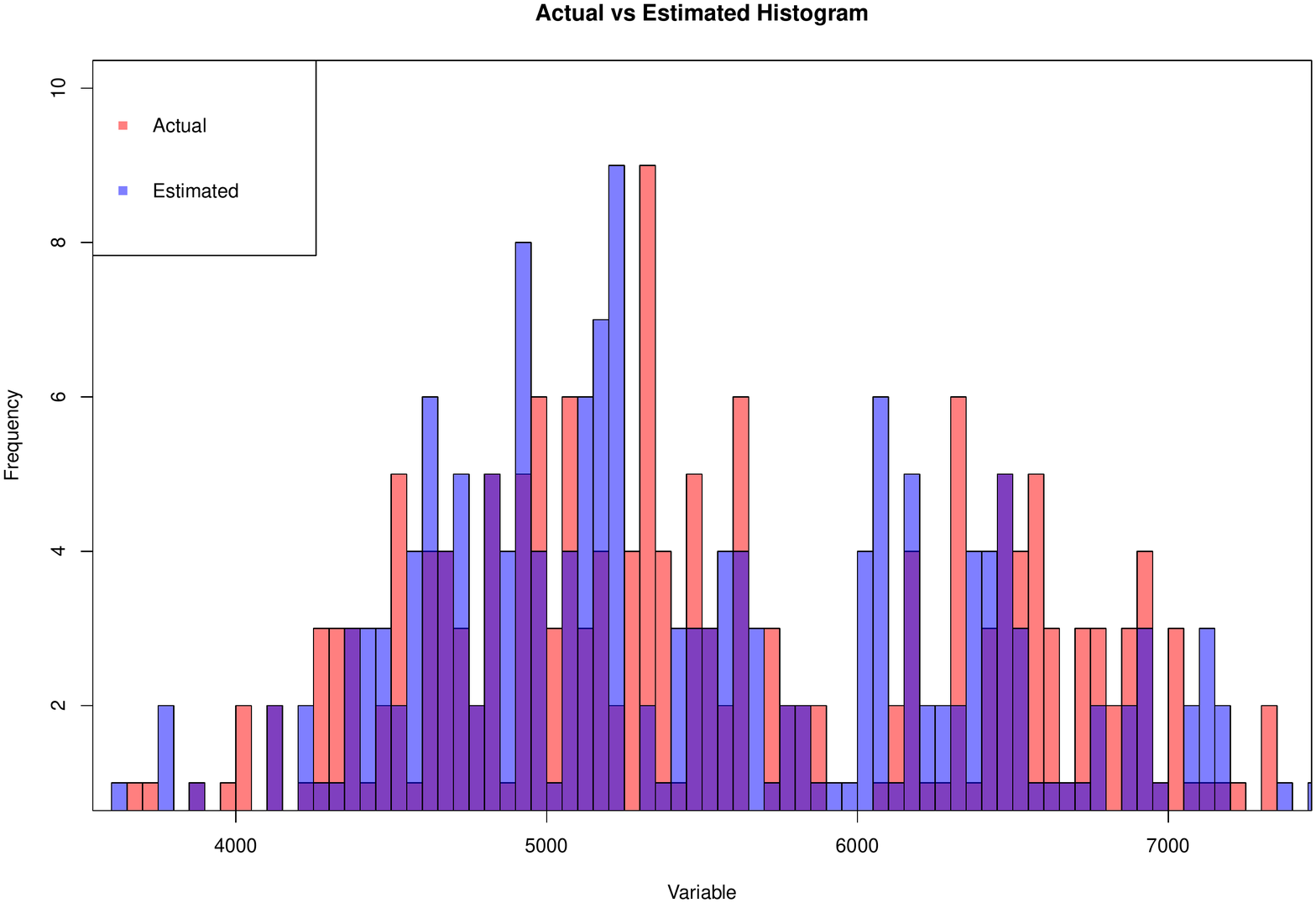}
	\caption{Histogram Actual \textit{vs.}~Estimated}
	\label{fig:hist2}
\end{figure}
Note that, however, this result stemmed from initial values derived from moments. In addition, we found that estimates vary much with different chosen initial values. For instance, if we begin with $\alpha_1^{(0)} = 150; \alpha_2^{(0)}  = 20; \beta^{(0)}  = 0.02.$ Convergence happened at the $308^{th}$ iteration at the same toleration level above. Nevertheless, corresponding estimates lie closer to the initial values.
\[
\alpha_1 = 141.98069243;  \alpha_2 = 28.02122426; \beta = 0.02773337, p = 0.5723845.
\]
As can be seen in Figure \ref{fig:hist3}, we still have a quite good fit visually. But we could already tell that it possibly does not provide a fitting as good as, if not worse than the previous one. 
\begin{figure}
	\center
	\includegraphics[scale=0.3]{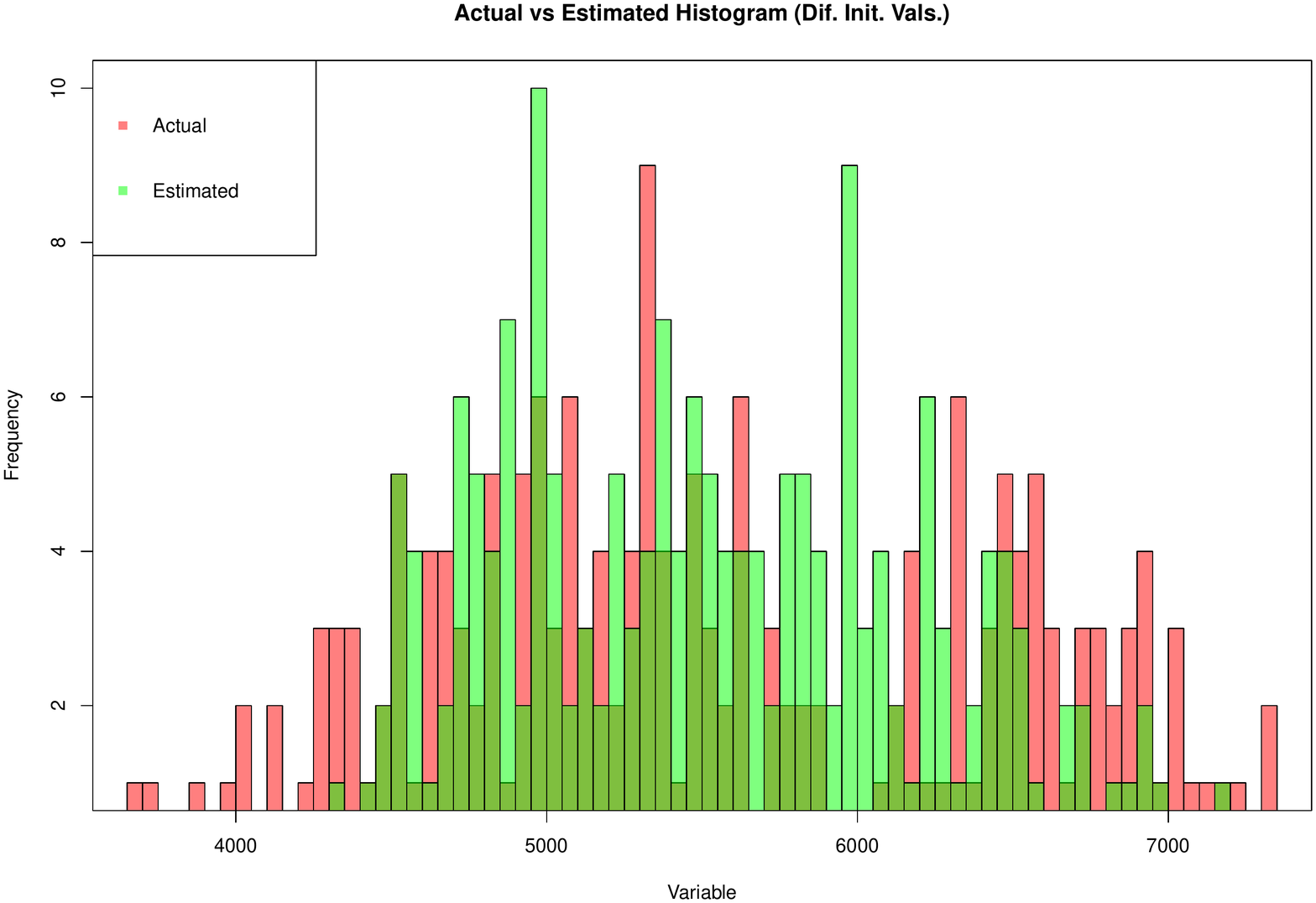}
	\caption{Histogram Actual \textit{vs.}~Estimated using a different set of init. vals}
	\label{fig:hist3}
\end{figure}

On the other hand, the Kolmogorov-Smirnov and Chi-square test results are summarised in Table~\ref{table:tests}. 
\begin{table}[t]
	\centering
	\begin{tabular}{|c|c|c|}
		\hline
		\thead{Reference Distribution} & \thead{Kolmogorov-Smirnov test} & \thead{Chi-square test}\\ \hline
		\makecell{NB1(97.55820446, 0.01978072) \\and NB2(127.70527118, 0.01978072)} & 0.3983 & 0.2769 \\ \hline
		\makecell{NB1(141.98069243, 0.02773337) \\and NB2(170.00191669, 0.02773337)} & 0.006767 & 0.0004998\\ \hline
	\end{tabular}
\caption{Summary of goodness-of-fit tests}
	\label{table:tests}
\end{table}
In the former case, both $ p $-values are greater than 25\%. That is to say, we cannot reject the null hypothesis that the observed data agree with the proposed model at 25\% significance level. On the contrary, neither tests present $ p $-values greater than 0.01 for the second example. Therefore, we could conclude the second model does not fit the data well even at 1\% significance level.
\subsubsection{Claim severity component}
We are now ready to implement the steps of EM algorithm for estimating parameters in the claim severity model. Before proceeding to the EM steps, let us look into the parameter $\nu$. In fact, $\nu$ is connected to the parameters in the claim frequency component. Therefore, once the sample of claim counts is given, we would get an estimated value for $\nu$. Recall from previous sections that,
\begin{equation}
\nu = p + (1-p)\cdot\frac{B(\alpha_1+1, \alpha_2)}{B(\alpha_1, \alpha_2)}\, .
\label{eq:nu}
\end{equation}
%
%
%
If we substitute in \eqref{eq:nu} the estimated values of $p,\, \alpha_1,\, \alpha_2$, we could obtain that the point estimation for $\nu$ would be $\nu = 0.9039196$. $\nu$ could then be treated as a constant onwards, and we follow an iteration procedure similar  to that presented in the previous subsection. As follows,
\begin{enumerate}
	\item We begin with an initially determined parameter values $\vartheta^{(0)} = \{\mu^{(0)}, \delta^{(0)}, \sigma^{(0)}\}$.
	\item E-Step 
	
	For the $(l+1)$-{th} iteration, $l = 0, 1, \ldots$,  we first seek for the expected value of the random value $U_i$, representing the latent variable, conditional on the observations together with the current parameter estimates $\vartheta^{(l)}$, i.e., estimates from the $l$-{th} iteration. Let $ \tau_i $ denote this expectation
	\begin{eqnarray*}
		\tau_i &=& \mathbb{E} [U_i| { Y,\vartheta^{(l)}}] = 1\times \mathbb{P}\{U_i = 1 | Y_i,\vartheta^{(l)}\} + 0 \times \mathbb{P}\{U_i = 0 | Y_i,\vartheta^{(l)}\} \\[0.15cm]
		&=&  \frac{\mathbb{P}\{U_i = 1, Y_i|\vartheta^{(l)}\}}{\mathbb{P}\{Y_i, \vartheta^{(l)}\}}
		=
		\frac{\nu \mu^{(l)} e^{-\mu^{(l)} y_i}}{\nu \mu^{(l)} e^{-\mu^{(l)} y_i} + (1-\nu)\frac{\delta^{(l)} {\sigma^{(l)}}^{\delta^{(l)}}}{(\sigma^{(l)}+y_i)^{\delta^{(l)}+1}}}
	\end{eqnarray*}
	Subsequently, we compute the expectation of its log-likelihood function with respect to the conditional distribution of $U$ given $Y$ under the current estimates $\vartheta^{(l)}$. Considering that we have $m^\ast$ independent observations, we have,
	\begin{eqnarray}
	Q(\vartheta |\vartheta^{(l)}) &= & \mathbb{E}_{U_i | Y_i, \vartheta^{(l)}}[\mathcal{L}(\vartheta | Y_i, U_i)]\label{eqn:clm_expectation}\\
	&=&\sum_{i}^{m^\ast}\tau_i \left[\log \nu + \log \mu^{(l)} -\mu^{(l)} y_i\right]\nonumber\\
	&& \hspace{0.25cm}+ \sum_{i}^{m^\ast} (1-\tau_i) \left[\log (1- \nu) + \log \delta^{(l)} +\delta^{(l)}\log \sigma^{(l)}-(\delta^{(l)}+1) \log(\sigma^{(l)}+y_i)\right]\nonumber\\[1.5mm]
	&=& \log \nu \sum_{i}\tau_i + \log\mu^{(l)} \sum_i \tau_i -\mu^{(l)} \sum_i \tau_i y_i  + \log (1- \nu) \sum_{i} (1-\tau_i) \nonumber\\
	&& + \log\delta^{(l)} \sum_{i} (1-\tau_i) +\delta^{(l)} \log \sigma^{(l)} \sum_{i} (1-\tau_i) - (\delta^{(l)}+1) \sum_{i} (1-\tau_i) \log(\sigma^{(l)}+ y_i).\nonumber
	\end{eqnarray}
	\item M-Step
	
	The maximisation step is set to find the parameter values that maximises \eqref{eqn:clm_expectation} and they are the estimates to be used in the next iteration, i.e.,
	\begin{equation}
	\vartheta^{(l+1)} =\operatornamewithlimits{argmax}_{\vartheta} Q(\vartheta |\vartheta^{(l)})
	\label{eqn:clm_max}
	\end{equation}
	In order to do this, we  take the gradient of $Q(\vartheta |\vartheta^{(l)})$, equate to zero and solve the system for $\{\mu, \delta, \sigma\}$ simultaneously.
	\begin{eqnarray}
	\frac{\partial Q}{\partial \mu} &=& \frac{\sum_i \tau_i}{\mu}-\sum_i \tau_i y_i =0;\nonumber\\
	\frac{\partial Q}{\partial \delta} &=& \frac{n-\sum_i\tau_i}{\delta} + n\log\sigma -\log \sigma\sum_i \tau_i -\sum_i (1-\tau_i) \log(\sigma+y_i) =0 ;\nonumber\\
	\frac{\partial Q}{\partial \sigma} &=& \frac{(n-\sum_i\tau_i)\delta}{\sigma} -  (\delta+1)\sum_i \frac{1-\tau_i}{\sigma+y_i}=0 \, . \nonumber
	\end{eqnarray}
	
The estimates for the subsequent iteration $(l+1)$-{th} are the solutions to these equations. 
	It is obvious that $\mu$ is independent from other parameters and could be solved directly with an explicit representation
	\[
	\mu^{(l+1)} = \frac{\sum_i \tau_i}{\sum_i \tau_i y_i}.
	\]
	For the other two parameters, we could only solve numerically. Again we apply the {\it nleqslv} package in R. As a consequence, we can obtain estimated parameters at this iteration $\vartheta^{(l+1)} = \{\mu^{(l+1)}, \delta^{(l+1)}, \sigma^{(l+1)}\}$.
	
	\item Plug $\vartheta^{(l+1)}$ into the $(l+2)$-{th} iteration and repeat the above steps until convergence.
\end{enumerate}

As we said, we do not have the claim severity data corresponding to the claim count one, we simulated those based on the claim counts data we used before. Then we have for claim counts in each period followed by a sequence of claim sizes randomly generated according to \eqref{eqn:clmdist} when $f(\cdot)$ and $g(\cdot)$ take $Exponential(\mu)$ and $Pareto(\delta, \sigma)$ forms, respectively.

Then, we implemented the algorithm and were able to achieve the desired parameters within reasonable amount of iterations. In Table~\ref{tab:est} we show our estimates against the predefined parameters. This verifies the effectiveness of our algorithm.
\begin{table}[h]
	\centering
	\begin{tabular}{|l|c|c|c|}
		\hline
		{\bf Parameters} & $\mu$ & $\delta$ & $\sigma$\\ \hline\hline
		Predefined & 1 & 2 & 0.5\\ \hline
		Initial value & 1.5 & 2.5 & 0.2\\ \hline
		After 1 iteration & 1.0018065 & 1.572045 & 0.1476012\\ \hline
		After 10 iterations & 0.9918480 & 2.011409 & 0.4177998\\ \hline
		After 110 iterations (at convergence) & 0.9999930 & 1.957350 & 0.4856047\\
		\hline
	\end{tabular}
\caption{Performance of estimation using EM algorithm on simulated claim severities without $\nu$}
	\label{tab:est}
\end{table}


As an alternative, we we have the option to estimate $\nu$ directly from claim size data. Implementing the EM algorithm to find this estimate $\hat{\nu}$ is very similar to that we applied to $p$ in Subsection~\ref{sec:emfreq}. Recall that the partial derivative equation with respect to $p$ in that subsection is independent from all other variables. A similar  situation exists here for $\nu$. We simply need to introduce one additional equation respecting ${\nu}$ estimation, note that it does not affect the remainder equations. As so, we find
\begin{equation}
\hat{\nu} = \frac{\sum_i \tau_i}{m^{\ast}}.
\end{equation}

After running the EM algorithm for this construction, we  obtain results in Table~\ref{tab:est2} and we can compare them with the ``true'' values.\\

\begin{table}[h]
	\centering
	\caption{Performance of estimation using EM algorithm on simulated claim severities with $\nu$}
	\begin{tabular}{|l|c|c|c|c|}
		\hline
		{\bf Parameters} & $\mu$ & $\delta$ & $\sigma$ & $\nu$\\ \hline\hline
		Predefined & 1 & 2 & 0.5 & 0.9039196\\ \hline
		Initial value & 1.5 & 2.5 & 0.2 & 0.9\\ \hline
		After 1 iteration & 1.000203 & 1.596950 & 0.1515525 & 0.9311160\\ \hline
		After 100 iterations & 1.005289 & 1.819941 & 0.3791028 & 0.9196332\\ \hline
		After 1552 iterations (at convergence) & 0.9991517 & 2.012381 & 0.4928861 & 0.904398 \\
		\hline
	\end{tabular}
	\label{tab:est2}
\end{table}

Notice that we have input the predefined value estimate for $\nu$ as the estimate from the claim count data. In this way, we could further consider the value estimated from this method, i.e., 0.904398 from Table~\ref{tab:est2}, to be a secondary estimate for $\nu$. Iteration reached convergence though at a much lower speed than before (intuitively, we could expect this). It still led to good estimation of all the parameters including $\nu$, converging closely to its ``true'' value, eventually.
\subsection{A short discussion}
To conclude, on the model and estimation procedure, this manuscript has carried out a develoment and and subsequent parameter estimation for the so-called unforeseeable risks discussed in \cite{li2015risk}. Precisely, for these risks, a probability $p$ has been assigned at mass point $\{0\}$ so that their corresponding counting process is distinguished from the classical one. Since we could only observe the entire claim counts from an insurance portfolio, this missing information could be estimated using the EM algorithm. Under certain assumptions for the distribution of heterogeneities within the portfolio (which was denoted as $\Lambda$ in this work), we could derive the randomness of the total claim counts given a fixed period to be a Negative Binomial mixture distribution. Thus, the likelihood function could be presented explicitly with which EM algorithm was implemented properly, based on a set of claim counts data for the third party liability insurance portfolio. However, the resulting estimates seem to be very sensitive to chosen initial values. Hence, we employed starting values which were computed via method of moments. Both the Kolmogorov-Smirnov and Chi-square tests suggest a good fit on the observed data.

We added to the initial model a study on missing information on the claim severity, assuming that it could bring some information, may be duplicate, on the unforeseeable stream. We could also apply a similar estimation procedure to that of the claim count one.  

According to the dataset considered here, the estimated probability of non-occurrence of claims per period is approximately 60\%. This is surprisingly larger than what we would expect, and we can argue about it. It may be due to the quality of the data set of course, but also note that it corresponds to the probability for a fixed time period which is a quarter here.  Nevertheless, it may show that many policies have not filed their claims to the system, yet. Also, Similar to the concerns usually put into the right tail extreme, i.e, extreme value theory, it could also be interesting to devote some particular atention to the left extreme. The latter one is what an insurer deals with on a daily basis, it could be related with the ``Bonus Hunge'' problem, common in motor insurance. It could be quite dangerous when someone does not show any claims before a large claim. This could be many small scratches unreported, leaving no history, that lead to a big accident. To detect the potential of policyholders moving from left towards right end tail would probably be more secure for an insurer. It also applies to introduction of new policies covering autonomous cars for instance where initial information is missing and accurate estimation of risks is needed. This study, however, serves as a starting tool to help identifying some of these risks.
\section{Combined parameter estimation and global likelihood}
In our model we could consider the calculation of Bayesian premia separately for claim counts and respective severities since we built a model where there is stochastic independence between these two quantities. However, it would be nice to compute a premium where we estimate parameters $\vartheta := (\alpha_1,\alpha_2, \beta, \mu, \delta, \sigma, p)$ altogether, using a likelihood where $N$ and $Y$ are jointly distributed. We will refer to it as the global likelihood function and it will serve as a direct extension of the current work.

In general, a random sample will be composed by a sequence of $m$ independent pairs of dependent observervations $(N_i, \overrightarrow{Y_i})$, where $N_i$ represents the number of claims in the $i$-{th} period and $\overrightarrow{Y_i} = (Y_{i1}, \ldots, Y_{in_i})$ is the corresponding sequence of claim severities. It is clear that if we consider in general that both claim counts and severities may bring information about each stream, foreseable or unforeseable, we should consider that stochastic dependence between $N_i$ and corresponding  $\overrightarrow{Y_i}$ is present. However, they are conditionaly independent, given $\Lambda=\Lambda_i\, $ ($i=1,2$) that is, for each individual stream we can consider the classical asumption of independence between claim counts and severities. 

Observed values are represented by corresponding lower case letters $(n_i, \overrightarrow{y_i})$. We also note that in each pair the dimension of vector $\overrightarrow{y_i}$ depend on the observed $n_i$. Now, we could write a likelihood function considering the joint random vector $(N, \overrightarrow{Y})$ assuming we have $m$ groups of observations and conditional independence of each $Y_{ij}$, $j = 1, \ldots, n_i$, for a given $i$:
\begin{eqnarray*}
	L(\vartheta | \mathbf{n, \overrightarrow{y}}) = \prod_{i=1}^m f_{N, \overrightarrow{Y}}(n_i, y_{i1}, \ldots, y_{in_i}) &=& \prod_{i = 1}^m f_{\overrightarrow{Y}|N}(\overrightarrow{y_i}| n_i) \mathbb{P}(n_i) = \prod_{i=1}^m \left\{\prod_{j=1}^{n_i} f_{\overrightarrow{Y}|N}(y_{ij}| n_i)\right\} \mathbb{P}(n_i)\,,
\end{eqnarray*}
where $\mathbb{P}(n_i)=\mathbb{P}(N=n_i)$, for simplification.

For our model, we can write the corresponding global log-likelihood function, denoted as $\mathcal{L}(\vartheta | \mathbf{n, \overrightarrow{y}})$. If  we let claim counts conform to Lemma~\ref{l:mix_nb} and claim severities follow \eqref{eqn:clmdist}, i.e.~$f(\cdot)$ and $g(\cdot)$ are respectively Exponential and Pareto densities, we have
{\small\begin{eqnarray*}
		\mathcal{L}(\vartheta | \mathbf{n, \overrightarrow{y}}) &=& \sum_{i=1}^m \log \mathbb{P}(n_i) + \sum_{i=1}^m \sum_{j=1}^{n_i} \log f_{\overrightarrow{Y}|N}(\overrightarrow{y_i}| n_i)\\
		&=& \sum_{i=1}^m \log\left\{p\binom{n+\alpha_1-1}{n}\left(\frac{\beta}{\beta + 1}\right)^{\alpha_1}\left(\frac{1}{\beta + 1}\right)^n \right.\\
		& & \;\;\; \;\;\; 
		\left. + (1-p)\binom{n+\alpha_1+\alpha_2-1}{n}\left(\frac{\beta}{\beta + 1}\right)^{\alpha_1+\alpha_2}\left(\frac{1}{\beta + 1}\right)^n\right\}
		\\
	    && 
		 \;\;\; \;\;\; \;\;\; \;\;\; +
		 \sum_{i=1}^m \sum_{j=1}^{n_i} \log \left\{ \nu \mu e^{-\mu y_{ij}} + (1-\nu) \frac{\delta \sigma^\delta}{(\sigma+ y_{ij})^{\delta + 1}}\right\}.
\end{eqnarray*}}

Instead of building a likelihood function based on a sample of observations of the pair $(N, \overrightarrow{Y})$, we can build it using interarrival time and corresponding severity, where each observation is a bivariate pair, denoted as $(T_j, Y_j)$, of dependent random varibles, in general. Let's denote the density and the distribution function of $T_j$ by $\phi(\cdot)$ and $\Phi(\cdot)$, respectively.

For finding the distribution of $(T_j)$ recall Lemma~\ref{l:mix_nb2} and that, conditional on $\Lambda =\lambda$, $\{N(t)\}$ is a  Poisson process with intensity $\lambda$. Then, given $\Lambda =\lambda$, conditional interarrival time $T_j|\Lambda =\lambda \frown Exponential (\lambda)$, with mean $\lambda^{-1}$. If Poisson parameter is an outcome of a random variable $\Lambda$ and is distributed as \eqref{eqn:gammamix}, we have that the unconditional distribution of $T_j$ is a mixture of two Pareto distributions whose density is given by 
	\begin{equation}
	\label{eq:mixpareto}
	\phi(t_j)=p\,\frac{\alpha_1\beta^{\alpha_1}}{(\beta+t_j)^{\alpha_1+1}}
	+ (1-p)\, \frac{(\alpha_1+\alpha_2)\beta^{\alpha_1+\alpha_2}}{(\beta+t_j)^{\alpha_1\alpha_2+1}}
	\end{equation}
		
On the other hand, the 
 density function of $Y_j$, conditional on a given $\Xi=\xi$, is given by the mixture, see~\eqref{eq:dfYxi},
	\begin{equation}
	\label{eq:mix_fg}
	h_\xi(y_j)= \xi f(y_j)+(1-\xi)g(y_j)\, ,
	\end{equation}
	and the unconditional density is given by \eqref{eqn:clmdist}. 
	 In particular, we assumed in our model that $f\frown exp(\mu)$, $g\frown Pareto(\delta,\sigma)$.

Recall Remark~\ref{rem:indep}, we know that
 with $\Xi=\xi$ fixed, $T_j$ and $Y_j$ are independent. A random sample of observations are made of independent pairs $(Y_j, T_j)$, each $j$ of dependent $Y_i$ and $T_i$. Each pair, $j$, has density function
\begin{eqnarray}
	\label{eq:Ltheta}
	\phi(y_i,t_i) & = & \int_{\xi}h_\xi(y_i)\phi(t_i)dA(\xi) \\
	& = & (1-p) \int_{\xi\neq 1}h_\xi(y_i)\phi(t_i)dA(\xi) + p\, h_1(y_i)
	\frac{\alpha_1\beta^{\alpha_1}}{(\beta+t_i)^{\alpha_1+1}} \,  \nonumber
\end{eqnarray}
 where $A(\xi)$ stands for the distribution function of $\Xi$. Note that that we separated the situations $\xi\neq 1$ and $\xi=1$. Also, the events are equivalent: $\{\Xi=1\}\Leftrightarrow \{ \Lambda^{(2)}=0\}$, so that $Pr\{\Xi=1\}=Pr \{ \Lambda^{(2)}=0\}=p$.

A short discussion on the distributions above, \eqref{eq:mixpareto} and \eqref{eq:mix_fg}, sounds apropriate. First, we consider the unconditional distribution of the arrival time (unconditional of $\Lambda$) however, we considered the conditional of the individual claim size $Y$, given $\Xi=\xi$. This ratio is given but that only means that $\Lambda^{(1)}$ is a given proportion of $\Lambda$, it remains random, it only means that the distribution of $\Lambda^{(1)}$ is a scaled distribution of that of $\Lambda$, or \textit{vice versa}. Looking at the density $h_\xi(y_i)$ we see that the split rate $\xi$ gives the weight for the foreseeable stream claim amount, and looking at the density $\phi(y_j)$ we could think that the probability $p$ gives a similar meaning regarding the claim count stream. It does not seem the case, it is not clear that the second part in the mixture represents the unforeseeable only.

Now, let's back to joint density~\eqref{eq:Ltheta}. We build the likelihood function over $m^\ast$ pairs of observations ($m^\ast$ may be different from sample size $m$ from random vector $(N, \overrightarrow{Y})$ above):
\begin{eqnarray*}
L(\vartheta) &= &\prod_{i=1}^{m^\ast}\Phi(y_i,t_i)=
\prod_{i=1}^{m^\ast} \int_{\xi}h_\xi(y_i)\phi(t_i)dA(\xi) \\
& = & \prod_{i=1}^{m^\ast} \left[
(1-p)\int_{\xi\neq 1} \left(\xi f(y_i)+(1-\xi)g(y_i)\right) \phi(t_i)Beta(\xi;\alpha_1,\alpha_2)d\xi +p\, f(y_i)\, \frac{\alpha_1\beta^{\alpha_1}}{(\beta+t_i)^{\alpha_1+1}}
\right] ,
\end{eqnarray*}
where $Beta(\xi;\alpha_1,\alpha_2)$ denotes the Beta density function of $\Xi$, given $\Xi\neq 1$, and $\frac{\alpha_1\beta{\alpha_1}}{(\beta+t_i)^{\alpha_1+1}}$ is the density $\phi(t_j)$ when $\xi=1$.
Remark that we cannot interchange the product ($\prod_{i = 1}^{m^\ast}$) and the integral. 
%
\bibliographystyle{plain}
\bibliography{refs}
\end{document}